\documentclass{amsart}\usepackage[a4paper, margin=3cm]{geometry}

\usepackage{amsfonts, amssymb, amsmath, eucal, verbatim, amsthm, amscd, enumerate}

\usepackage{graphicx}
\usepackage{framed}
\usepackage{tikz}

\usepackage{ascmac}

\newtheorem{theorem}{Theorem}[section]
\newtheorem{lemma}[theorem]{Lemma}
\newtheorem{proposition}[theorem]{Proposition}
\newtheorem{conjecture}[theorem]{Conjecture}

\newtheorem{corollary}[theorem]{Corollary}
\theoremstyle{definition}

\theoremstyle{remark}
\newtheorem*{remark}{Remark}

\newcommand{\vertiii}[1]{{\left\vert\kern-0.25ex\left\vert\kern-0.25ex\left\vert #1
    \right\vert\kern-0.25ex\right\vert\kern-0.25ex\right\vert}}

\numberwithin{equation}{section}

\usepackage{setspace}

\def\1{\textbf{\rm 1}}

\def\Xint#1{\mathchoice
{\XXint\displaystyle\textstyle{#1}}%
{\XXint\textstyle\scriptstyle{#1}}%
{\XXint\scriptstyle\scriptscriptstyle{#1}}%
{\XXint\scriptscriptstyle\scriptscriptstyle{#1}}%
\!\int}
\def\XXint#1#2#3{{\setbox0=\hbox{$#1{#2#3}{\int}$}
\vcenter{\hbox{$#2#3$}}\kern-.5\wd0}}

\def\dashint{\Xint-}

\begin{document}

\date{\today}

\author[Jonathan Bennett]{Jonathan Bennett}
\address[Jonathan Bennett]{School of Mathematics, The Watson Building, University of Birmingham, Edgbaston,
Birmingham, B15 2TT, England}\email{j.bennett@bham.ac.uk}
\author[Shohei Nakamura]{Shohei Nakamura}\thanks{This work was partially supported by the European Research Council [grant
number 307617] (Bennett), and by JSPS Grant-in-Aid for JSPS Research Fellow no. 17J01766 and JSPS Overseas Challenge Program for Young Researchers (Nakamura).}
\address[Shohei Nakamura]{Department of Mathematics and Information Sciences, Tokyo Metropolitan University,
1-1 Minami-Ohsawa, Hachioji, Tokyo, 192-0397, Japan}
\email{nakamura-shouhei@ed.tmu.ac.jp}

\title{Tomography bounds for the Fourier extension operator and applications}

\begin{abstract}
We explore the extent to which the Fourier transform of an $L^p$ density supported on the sphere in $\mathbb{R}^n$ can have large mass on affine subspaces, placing particular emphasis on lines and hyperplanes. This involves establishing
bounds on quantities of the form $X(|\widehat{gd\sigma}|^2)$ and $\mathcal{R}(|\widehat{gd\sigma}|^2)$, where $X$ and $\mathcal{R}$ denote the X-ray and Radon transforms respectively; here $d\sigma$ denotes Lebesgue measure on the unit sphere $\mathbb{S}^{n-1}$, and $g\in L^p(\mathbb{S}^{n-1})$. We also identify some conjectural bounds of this type that sit between the classical Fourier restriction and Kakeya conjectures. Finally we provide some applications of such tomography bounds to the theory of weighted norm inequalities for $\widehat{gd\sigma}$, establishing some natural variants of conjectures of Stein and Mizohata--Takeuchi from the 1970s. Our approach, which has its origins in work of Planchon and Vega, exploits cancellation via Plancherel's theorem on affine subspaces, avoiding the conventional use of wave-packet and stationary-phase methods.
\end{abstract}

\maketitle

\section{Introduction and statements of results}
The purpose of this paper is to investigate ways in which basic ideas from tomography may be used to further develop our understanding of the Fourier extension operator from euclidean harmonic analysis. We begin this section with a brief introduction to the necessary aspects of the classical theory of the Fourier extension operator (known as restriction theory), and then proceed to present our results. These naturally divide into three parts. The first and second are exploratory, and expose a natural interplay between the Fourier extension operator and the Radon and X-ray transforms 
(Sections 1.2 and 1.3 respectively). The third part (Section 1.4) is driven by the prospect of applications to existing problems in restriction theory, and culminates in some progress on well-known conjectures of Stein and Mizohata--Takeuchi from the 1970s. Our work takes its inspiration from that of Planchon and Vega in \cite{PV}.
\subsection{Background: the Fourier extension operator} A fundamental objective of modern harmonic analysis is to understand the integrability properties of Fourier transforms of densities supported on ``curved" submanifolds of $\mathbb{R}^n$. The primordial example of such a submanifold, and the subject of this paper, is the unit sphere $\mathbb{S}^{n-1}$, which serves as a model for quite general smooth compact submanifolds of nonvanishing gaussian curvature.
Questions of this type are phrased in terms of the \emph{Fourier extension operator} $$g\mapsto\widehat{gd\sigma},$$ where
\begin{equation*}
\widehat{gd\sigma}(x)=\int_{\mathbb{S}^{n-1}}e^{ix\cdot\xi}g(\xi)d\sigma(\xi).
\end{equation*}
Here $d\sigma$ denotes surface measure on $\mathbb{S}^{n-1}$, $x\in\mathbb{R}^n$ and $g\in L^p(\mathbb{R}^n)$ for some $p\geq 1$. The extension operator is sometimes referred to as the \emph{adjoint Fourier restriction operator} since its (formal) adjoint is the mapping $$f\mapsto\widehat{f}\Bigl|_{\mathbb{S}^{n-1}}.$$
The celebrated \emph{restriction conjecture} states that
\begin{equation}\label{restconj}
\|\widehat{gd\sigma}\|_{L^q(\mathbb{R}^n)}\lesssim\|g\|_{L^p(\mathbb{S}^{n-1})}
\end{equation}
whenever
\begin{equation}\label{restpq}\frac{1}{q}<\frac{n-1}{2n}\;\mbox{ and }\; \frac{1}{q}\leq\frac{n-1}{n+1}\frac{1}{p'}.\end{equation}
The restriction conjecture has been verified in dimension $n=2$ (C. Fefferman and Stein \cite{Feff}, \cite{BigStein}; see also Zygmund \cite{Z}), and there has been considerable progress in higher dimensions in recent years (see for example \cite{HR} and \cite{Sto} for further discussion and context).
The necessity of the conditions \eqref{restpq} is straightforward to verify with simple examples. In particular the condition $\frac{1}{q}<\frac{n-1}{2n}$ amounts to the assertion that \eqref{restconj} holds with $g\equiv 1$. This is immediately apparent from the observation that
\begin{equation}\label{statphase}
|\widehat{\sigma}(x)|=\Bigl|\int_{\mathbb{S}^{n-1}}e^{ix\cdot\xi}d\sigma(\xi)\Bigr|\sim (1+|x|)^{-\frac{n-1}2}
\end{equation}
on a large portion of $\mathbb{R}^n$. This well-known bound follows from the method of stationary phase -- see \cite{Watson} or \cite{BigStein} for example. Accordingly, it is also conjectured that an endpoint inequality of the form
\begin{equation}\label{restconjend}
\|\widehat{gd\sigma}\|_{L^{\frac{2n}{n-1}}(B_R)}\lesssim_\varepsilon R^\varepsilon\|g\|_{L^{\frac{2n}{n-1}}(\mathbb{S}^{n-1})}
\end{equation}
holds for all $\varepsilon>0$; here $B_R$ denotes the ball of radius $R$ centred at the origin. It is well-known that \eqref{restconjend} for all $\varepsilon>0$, is equivalent to the restriction conjecture as stated above; see \cite{Tao-Bochner}.

\subsection{Radon transform bounds}\label{sub2}
Naively at least, the example $g\equiv 1$ above suggests that $L^2$ (rather than $L^{\frac{2n}{n-1}}$) is critical if we integrate on \emph{hyperplanes} (rather than the whole of $\mathbb{R}^n$). In other words, it seems natural to seek bounds on the quantities
\begin{equation}\label{compRad}
\mathcal{R}(|\widehat{gd\sigma}|^2)\;\;\;\mbox{ and }\;\;\;\mathcal{R}(\1_R|\widehat{gd\sigma}|^2),
\end{equation}
where $\mathcal{R}$ denotes the \emph{Radon transform},
$$
\mathcal{R}f(\omega,t):=\int_{x\cdot\omega=t}f(x)d\lambda_{\omega,t}(x).
$$
Here $(\omega,t)\in\mathbb{S}^{n-1}\times\mathbb{R}$ and the measure $d\lambda_{\omega,t}(x)=\delta(x\cdot\omega-t)dx$ is Lebesgue measure on the hyperplane $\{x\in\mathbb{R}^n:x\cdot\omega=t\}$.

The quantities \eqref{compRad} turn out to be very natural from other points of view. In particular, elementary considerations reveal that $\mathcal{R}$ is often unable to distinguish between $|\widehat{gd\sigma}|^2$ and $X_0^*(|g|^2)$, where $X_0$ denotes the \emph{restricted X-ray transform}
\begin{equation}\label{X0}
X_0f(\omega)=\int_{\mathbb{R}}f(s\omega)ds; \;\;\omega\in\mathbb{S}^{n-1}.
\end{equation}
It should be noticed that
\begin{equation}\label{exact}
X_0^*f(x)=|x|^{-(n-1)}(f({x}/{|x|})+f(-{x}/{|x|})),
\end{equation}
and so $\mathcal{R}X_0^*f(\omega,t)$ may be infinite unless the support of $f$ is contained in $\{x\in\mathbb{S}^{n-1}:x\cdot\omega\not=0\}$.
\begin{theorem}\label{planch}
For each $\delta\geq 0$, $f\in L^1(\mathbb{S}^{n-1})$ and $\omega\in\mathbb{S}^{n-1}$, let
$$
T_\delta f(\omega)=\int_{\mathbb{S}^{n-1}}\frac{f(x)}{|x\cdot\omega|+\delta}d\sigma(x).$$
Then,
\begin{enumerate}
\item
for any $\omega \in \mathbb{S}^{n-1}$, and $g\in L^2(\mathbb{S}^{n-1})$ supported in $\{x \in \mathbb{S}^{n-1}: x\cdot\omega > 0  \}$,
\begin{equation}\label{id1}
\mathcal{R}(|\widehat{gd\sigma}|^2)(\omega,t) = \mathcal{R}X_0^*(|g|^2)(\omega,t)=T_{0}(|g|^2)(\omega)
\end{equation}
for all $t\not=0$,
and
\item
\begin{equation}\label{id3}
\mathcal{R}(\1_R|\widehat{gd\sigma}|^2)(\omega,t)\lesssim T_{1/R}(|g|^2)(\omega),
\end{equation}
uniformly in $(\omega,t)\in\mathbb{S}^{n-1}\times\mathbb{R}$ and $R>0$.
\end{enumerate}
\end{theorem}

By symmetry, the identity \eqref{id1} also holds for $g$ supported in the ``lower" hemisphere $\{x \in \mathbb{S}^{n-1}: x\cdot\omega < 0  \}$.
We remark that the operator $T_0$ appearing in Theorem \ref{planch} is a variant of the spherical Radon (also known as Funk) transform
\begin{equation}\label{MCop0}
A_0f(\omega)=\int_{\mathbb{S}^{n-1}}f(x)\delta(x\cdot\omega)d\sigma(x).
\end{equation}
However, $T_0$ is more singular than $A_0$ from certain points of view. For example, $T_01$ is identically infinite, recalling the need for some care in interpreting \eqref{id1}. As a result, no Lebesgue space bounds on $\mathcal{R}(|\widehat{gd\sigma}|^2)$ are possible. As a substitute, we have the following near-uniform bounds on $\mathcal{R}(\1_R|\widehat{gd\sigma}|^2)$:
\begin{theorem}\label{t:main}
If
\begin{equation}\label{expoRad}
\ p\ge2,\ \frac{n-2}{2} + \frac1{2q} \ge \frac{n-1}{p},\ \frac{n-1}{q} \ge \frac{2}{p},
\end{equation}
then
\begin{equation}\label{e:radon-extension}
\big\| \mathcal{R}(1_{B_R}|\widehat{gd\sigma}|^2) \big\|_{L^q_\omega L^\infty_t} \lesssim \log(R) \| g \|_{L^p(\mathbb{S}^{n-1})}^2
\end{equation}
for all $R>0$.
\end{theorem}
Several remarks are in order. Firstly, the $L^\infty$ norm in $t$ is necessary, as may be seen quickly by considering the case $g\equiv 1$. This is closely related to the simple observation that $\mathcal{R}(|\widehat{gd\sigma}|^2)(\omega,t)$ is independent of $t$ for certain $g$ -- see Theorem \ref{planch}. Secondly, the range of exponents in \eqref{expoRad} is best-possible in the sense that the logarithmic growth must be replaced with power growth outside of this range.
Finally, the power of the logarithm in \eqref{e:radon-extension} is also best-possible. Our proof of Theorem \ref{t:main} will follow from \eqref{id3} combined with sharp bounds on the operator $T_\delta$. As may be expected given the logarithmic growth in $R$, these bounds on $T_\delta$ will follow from uniform bounds on the ``uncentred" spherical Radon transforms
\begin{equation}\label{MCopt}
A_tf(\omega)=\int_{\mathbb{S}^{n-1}}f(x)d\sigma_{\omega,t}(x)
\end{equation}
for small $t$; here $d\sigma_{\omega,t}(x)=\delta(x\cdot\omega-t)d\sigma(x)$. Several Lebesgue space estimates for these operators were considered by Christ in \cite{Christ84}, and our proof of Theorem \ref{t:main} involves only modest additions to his results.

It should be remarked that Theorem \ref{t:main} contains
Lebesgue space bounds on the composition \eqref{compRad} that are well beyond the scope of the restriction conjecture and possible estimates for the Radon transform -- the clearest example being the case $p=q=\infty$. We also note that for $n>2$, Theorem \ref{t:main} has as an endpoint the inequality
\begin{equation}\label{e:EndRadon}
\big\| \mathcal{R}(1_{B_R}|\widehat{gd\sigma}|^2) \big\|_{L^n_\omega L^\infty_t} \lesssim \log(R) \| g \|_{L^{\frac{2n}{n-1}}(\mathbb{S}^{n-1})}^2,
\end{equation}
which would follow (up to a factor of $R^\varepsilon$) from the conjectured endpoint restriction inequality \eqref{restconjend}, combined with a (missing) endpoint estimate for the Radon transform (see \cite{Oberlin-Stein}). Such ``improvements" are to be expected as the composition $\mathcal{R}X_0^*$ is much less singular than either of its factors.

The estimate \eqref{e:EndRadon} provides us with an opportunity to draw attention to the potential for ideas from tomography to be effective in addressing existing problems in restriction theory. By the inversion formula for the Radon transform, $f=c_n(-\Delta)^{\frac{n-1}{2}}\mathcal{R}^*\mathcal{R}f$, which holds for a suitably regular function $f$ on $\mathbb{R}^n$ and constant $c_n$, we may write
$$
|\widehat{gd\sigma}|^2\gamma_R=c_n(-\Delta)^{\frac{n-1}{2}}\mathcal{R}^*(\mathcal{R}(|\widehat{gd\sigma}|^2\gamma_R)),
$$
where $\gamma_R$ is a smooth bump function adapted to $B_R$. Hence by \eqref{e:EndRadon}, the restriction conjecture \eqref{restconjend} would follow if we knew that $(-\Delta)^{\frac{n-1}{2}}\mathcal{R}^*: L^n_\omega L^\infty_t\rightarrow L^{\frac{n}{n-1}}(B_R)$, with bound at most $O(R^\varepsilon)$. Unsurprisingly this is easily seen to not be the case in any dimension. However, there are precedents for this sort of approach to problems in the wider restriction theory -- see the forthcoming Section \ref{sub4} for further discussion and applications.

\subsection{X-ray transform bounds}\label{sub3}
As we have discussed, our motivation for considering integrals of $|\widehat{gd\sigma}|^2$ on hyperplanes comes from integrability considerations relating to examples that generate the conditions \eqref{restpq}. Of course the exponent $2$ ceases to be critical in this regard if we instead consider integrals on \emph{lines}. If one is prepared to sacrifice the obvious advantages of $L^2$ line integrals, one is naturally led to
look for bounds on
\begin{equation}\label{compX}
X(|\widehat{gd\sigma}|^{\frac{2}{n-1}})\;\;\;\mbox{ or }\;\;\;X(\1_R|\widehat{gd\sigma}|^{\frac{2}{n-1}}),
\end{equation}
where $X$ denotes the \textit{X-ray transform}
\begin{equation}\label{Xray}
Xf(\omega,v)=\int_{\mathbb{R}}f(v+s\omega)ds.
\end{equation}
Here $\omega\in\mathbb{S}^{n-1}$ and $v\in\langle\omega\rangle^\perp$ parametrise the manifold $\mathcal{M}_{1,n}$ of all doubly-infinite lines in $\mathbb{R}^n$ in the natural way.
In this setting there is a close conjectural analogue of the endpoint estimate \eqref{e:EndRadon} that sits between the restriction and Kakeya conjectures; see Section \ref{AP1} for a statement of the latter.
\begin{conjecture}\label{conj:Xrest}
For every $\varepsilon>0$ there is a constant $C_\varepsilon<\infty$ such that
\begin{equation}\label{e:2/n-1}
\big\| X( \1_{B_R} |\widehat{gd\sigma}|^{\frac{2}{n-1}} ) \big\|_{L^n_\omega L^\infty_v} \leq C_\varepsilon R^\varepsilon \|g\|_{L^{\frac{2n}{n-1}}(\mathbb{S}^{n-1})}^{\frac{2}{n-1}}
\end{equation}
for all $R>0$.
\end{conjecture}
\begin{proposition}\label{prop:implication}
$$
{\rm Restriction\; Conjecture} \;\Rightarrow \;{\rm Conjecture\; \ref{conj:Xrest} }\;\Rightarrow \;{\rm Kakeya\; Maximal\; Conjecture}.
$$
\end{proposition}
Although \eqref{e:EndRadon} and \eqref{e:2/n-1} are very similar, it is of course the quadratic character of the former that makes it more tractable.
However, despite the exponent $2$ appearing to be subcritical in the context of line integrals, it does turn out to be rather natural to consider $X(|\widehat{gd\sigma}|^2)$, as the following elementary result illustrates (see also the forthcoming results in Section \ref{sub4}).
\begin{theorem}\label{planchX} For $f\in L^1(\mathbb{S}^{n-1})$ and $\omega\in\mathbb{S}^{n-1}$, let
$$
Sf(\omega)=\left(\int_{-1}^1(A_tf(\omega))^2dt\right)^{\frac12}.
$$
Then, for any $v\in \langle \omega\rangle^\perp$,
\begin{equation}\label{idX1}
X(|\widehat{gd\sigma}|^2)(\omega,v)=2\pi\int_{-1}^1|\widehat{gd\sigma}_{\omega,t}(v)|^2dt,
\end{equation}
and
\begin{equation}\label{idX2}
\sup_{v\in\langle\omega\rangle^\perp}X(|\widehat{gd\sigma}|^2)(\omega,v)\leq X_0(|\widehat{|g|d\sigma}|^2)(\omega)=2\pi S(|g|)(\omega)^2,
\end{equation}
with equality if $g$ is single-signed. Here $A_t$ and $X_0$ are given by \eqref{MCopt} and \eqref{X0} respectively.
\end{theorem}
Theorem \ref{planchX} suggests looking for X-ray estimates of the form
\begin{equation}\label{e:r=infty}
\big\| X( |\widehat{gd\sigma}|^2) \|_{L^q_\omega L^\infty_v} \lesssim \| g \|_{L^p(\mathbb{S}^{n-1})}^2.
\end{equation}
For $n=3$ at least, Theorem \ref{planchX} allows us to provide a complete picture for the inequality \eqref{e:r=infty}.
\begin{theorem}\label{t:n=3,r=infty}
Suppose $n=3$ and $p,q\ge1$.  Then
\begin{equation}\label{e:n=3r=infty}
\big\| X( |\widehat{gd\sigma}|^2) \|_{L^q_\omega L^\infty_v} \lesssim \| g \|_{L^p(\mathbb{S}^{2})}^2
\end{equation}
holds if and only if
\begin{equation}\label{e:n3rinfty-nec}
\frac1p \le\min\left\{\frac12 + \frac1{2q}, \frac34\right\},\quad \left(\frac1p,\frac1q\right) \neq \left(\frac34,\frac12\right).
\end{equation}
\end{theorem}
Theorem \ref{t:n=3,r=infty} also contains estimates that lie far beyond what may be obtained by applying known, or indeed possible, estimates for the extension operator and X-ray transform.
A simple example is the case $(p,q)=(2,\infty)$. Of course $X:L^p\not\rightarrow L^\infty$ for any $p$, and so no bound of this type may be deduced from the restriction conjecture.
Such ``improvements" have a simple heuristic explanation based on the standard wavepacket decomposition of the extension operator and a well-known (probabilistic) link between the extension operator and the $X$-ray transform.
This link, which famously connects the restriction conjecture to the Kakeya conjecture, reveals that $|\widehat{gd\sigma}|^2$ is, in some average sense, comparable to $X^*(|h|^2)$ for some function $h:\mathcal{M}_{1,n}\rightarrow\mathbb{C}$ formed from the wavepacket decomposition of $g$.  We refer the reader to \cite{BCSS} or \cite{TaoNotes} for some clarification of these heuristics. The main point here is that the composition $XX^*$ is much less singular than $X$.
We remark that the case $(p,q)=(2,\infty)$ mentioned here is straightforward to prove in the sharp form
\begin{equation}\label{infX}
\|X(|\widehat{gd\sigma}|^2)\|_{L^\infty}\leq 2\pi^2\|g\|_{L^2(\mathbb{S}^2)}^2,
\end{equation} where constant functions are among the extremisers -- see Section \ref{Sec6}.

The key ingredient in our proof of Theorem \ref{t:n=3,r=infty} is a power-weighted $L^p$ extension inequality considered by Bloom and Sampson \cite{BS92}. Our argument will require an endpoint case left open in \cite{BS92}, which we present in the appendix.

\subsection{Applications of tomography bounds to restriction theory}\label{sub4}
The basic principle of X-ray tomography is captured by the well-known inversion formula,
\begin{equation}\label{inversion}
f=c_nX^*(-\Delta_v)^{\frac12}Xf,
\end{equation}
or the closely-related fact that $c_n^{1/2}(-\Delta_v)^{1/4}X$ is an isometry between $L^2(\mathbb{R}^n)$ and $L^2(\mathcal{M}_{1,n})$ for a certain dimensional constant $c_n$. We might therefore expect that estimates on $X(|\widehat{gd\sigma}|^2)$, or its variants, may be used to address existing problems in restriction theory. A precedent for this approach may be found in the work of Planchon and Vega \cite{PV}, where certain sharp Strichartz estimates for the Schr\"odinger equation are obtained from identities involving the Radon transform of $|u(\cdot,t)|^2$, where $u$ is a solution to the free time-dependent Schr\"odinger equation; see also Beltran and Vega \cite{BelVega}, where their X-ray analysis is related to the recent sharp Stein--Tomas restriction theorem of Foschi.

A particularly compelling candidate for such an application is a conjectural weighted inequality attributed to Mizohata and Takeuchi \cite{MT} (see also \cite{BRV}), which states that
\begin{equation}\label{MizTak}
\int_{\mathbb{R}^n} |\widehat{gd\sigma}|^2\, w  \lesssim \| Xw \|_{L^{\infty}} \int_{\mathbb{S}^{n-1}}|g|^2
\end{equation}
for any weight function $w$ on $\mathbb{R}^n$.
This conjecture dates back to the 1970s, and remains unknown for general weights even for $n=2$ (if $w$ is radial then it is known, being equivalent to a certain uniform eigenvalue estimate involving Bessel functions -- see \cite{BRV}, \cite{CS} or \cite{BBC3}).
Motivated by the numerology of the standard Sobolev embeddings into $L^\infty$, it is perhaps natural to embed \eqref{MizTak} in a family of inequalities resembling
\begin{equation}\label{e:wMizTak}
\int_{\mathbb{R}^n} |\widehat{gd\sigma}|^2\, w  \lesssim \big\| (-\Delta_v)^{\frac{n-1}{2q}} Xw \big\|_{L^{\infty}_\omega L^q_v(\mathcal{M}_{1,n})} \int_{\mathbb{S}^{n-1}}|g|^2,
\end{equation}
where $1\leq q\leq\infty$. [Our reasoning here is of course merely heuristic -- strictly speaking the Sobolev embedding should involve the inhomogeneous derivative $(1-\Delta)$ raised to a power strictly larger than $\frac{n-1}{2q}$.]
Of course \eqref{MizTak} is just the case $q=\infty$, and so it is natural to try to establish a form of \eqref{e:wMizTak} for $q$ as large as possible.
To this end we may use the aforementioned fact that $c_n^{1/2}(-\Delta_v)^{1/4}X$ is an isometry between $L^2(\mathbb{R}^n)$ and $L^2(\mathcal{M}_{1,n})$ to write
\begin{eqnarray}\label{basicidea}\begin{aligned}
\int_{\mathbb{R}^n}|\widehat{gd\sigma}|^2w&=c_n\left\langle(-\Delta_v)^{\frac14}X(|\widehat{gd\sigma}|^2),(-\Delta_v)^{\frac14}Xw\right\rangle_{L^2(\mathcal{M}_{1,n})}\\&=c_n\left\langle(-\Delta_v)^{\frac1{2}(1- \frac{n-1}{q})}X(|\widehat{gd\sigma}|^2),(-\Delta_v)^{\frac{n-1}{2q}}Xw\right\rangle_{L^2(\mathcal{M}_{1,n})}\end{aligned}
\end{eqnarray}
for all $1\leq q\leq\infty$. An application of H\"older's inequality now leads to the bound
\begin{equation}
\int_{\mathbb{R}^n}|\widehat{gd\sigma}|^2w\lesssim \big\|(-\Delta_v)^{\frac{n-1}{2q}}Xw \big\|_{L^\infty_\omega L^q_v} \big\|(-\Delta_v)^{\frac12(1-\frac{n-1}q)}X(|\widehat{gd\sigma}|^2) \big\|_{L^1_\omega L^{q'}_v}.
\end{equation}
The tentative estimate \eqref{e:wMizTak} may therefore be reduced to
\begin{equation}\label{e:Reduced0}
\big\|(-\Delta_v)^{\frac12(1-\frac{n-1}q)}X(|\widehat{gd\sigma}|^2) \big\|_{L^1_\omega L^{q'}_v(\mathcal{M}_{1,n})} \lesssim \|g\|_{L^2(\mathbb{S}^{n-1})}^2.
\end{equation}
In order to ensure finiteness in \eqref{e:wMizTak} and \eqref{e:Reduced0} we consider here the validity of the local variant
\begin{equation}\label{e:wMizTak2}
\int_{B_R} |\widehat{gd\sigma}|^2\, w  \lesssim R^\varepsilon \big\| (-\Delta_v)^{\frac{n-1}{2q}} Xw \big\|_{L^{\infty}_\omega L^q_v(\mathcal{M}_{1,n})} \int_{\mathbb{S}^{n-1}}|g|^2,
\end{equation}
formulated in the spirit of \eqref{restconjend}.
Arguing as above, \eqref{e:wMizTak2} would follow from the estimate
\begin{equation}\label{e:Reduced00}
\big\|(-\Delta_v)^{\frac12(1-\frac{n-1}q)}X(\gamma_R|\widehat{gd\sigma}|^2)\big\|_{L^1_\omega L^{q'}_v(\mathcal{M}_{1,n})} \lesssim R^\varepsilon\|g\|_{L^2(\mathbb{S}^{n-1})}^2,
\end{equation}
where $\gamma_R$ is a smooth bump function adapted to $B_R$ (satisfying certain technical conditions that we clarify in Section \ref{Sect:appl}).
The first thing to notice is that \eqref{e:Reduced00}, and hence \eqref{e:wMizTak2}, with $q=1$ and $n=2$ is a direct consequence of Theorem \ref{t:main}. Our main result here states that, for $n=2$, the exponent $q$ may be pushed up to $2$.
\begin{theorem}\label{t:wMizTak}
Let $n=2$. Then \eqref{e:Reduced00}, and hence \eqref{e:wMizTak2}, holds true as long as $1\le q\le 2$. Moreover,
\begin{equation}\label{e:Reduced}
\big\|(-\Delta_v)^{\frac{1}{4}}X(\gamma_R|\widehat{gd\sigma}|^2) \big\|_{L^1_\omega L^{2}_v(\mathcal{M}_{1,2})} \lesssim \log R\|g\|_{L^2(\mathbb{S}^{1})}^2,
\end{equation}
and hence
\begin{equation}\label{e:wMizTak22}
\int_{B_R} |\widehat{gd\sigma}|^2\, w \lesssim \log R\big\| (-\Delta_v)^{\frac{1}{4}} Xw \big\|_{L^{\infty}_\omega L^2_v(\mathcal{M}_{1,2})} \int_{\mathbb{S}^{1}}|g|^2.
\end{equation}\end{theorem}
The first remark to make is that the tomography reduction presented above, as it stands at least, fails to establish the Mizohata--Takeuchi conjecture \eqref{MizTak}, even with a growth factor of the form $R^\varepsilon$ in the truncation parameter $R$. Specifically, the estimate \eqref{e:Reduced00} is easily seen to fail for $n=2$ when $q>2$, even for the function $g\equiv 1$. In this sense the estimate \eqref{e:Reduced} in the statement of Theorem \ref{t:wMizTak} is best-possible.


We shall reduce Theorem \ref{t:wMizTak} to a stronger two-weighted estimate for $\widehat{gd\sigma}$ in the spirit of a well-known conjecture of Stein \cite{S}; see also C\'ordoba \cite{Co} and Carbery--Soria--Vargas \cite{CSV} for variants of this. In the context of the extension operator, Stein's conjecture takes the form
\begin{equation}\label{Stconj}
\int_{B_R}|\widehat{gd\sigma}|^2w\lesssim\int_{\mathbb{S}^{n-1}}|g|^2\mathcal{M}w,
\end{equation}
where $\mathcal{M}$ is (possibly a variant of) the Kakeya-type maximal operator
\begin{equation}\label{steinX}
\mathcal{M}f(\omega)=\sup_{v\in\langle\omega\rangle^\perp}X|f|(\omega,v),
\end{equation}
or equivalently,
$$
\mathcal{M}_{R}f(\omega):=\sup_{T\parallel\omega}\int_T|f|,
$$
where in this last expression the supremum is taken over all $1$-neighbourhoods of line segments in $\mathbb{R}^n$ of length $R$, parallel to the direction $\omega\in\mathbb{S}^{n-1}$. Of course \eqref{Stconj}, with maximal operator given by \eqref{steinX}, implies \eqref{MizTak} since $\|\mathcal{M}w\|_\infty=\|Xw\|_\infty$. The proposed inequality \eqref{Stconj} is intended to clarify the relationship between the restriction and Kakeya conjectures, allowing the conjectural $L^p-L^q$ bounds for the extension operator to follow from those for $\mathcal{M}_R$. Specifically, it is straightforward to verify that the Kakeya (maximal) conjecture
\begin{equation}\label{kakeyaconjecturescaled}
\|\mathcal{M}_Rf\|_{L^n(\mathbb{S}^{n-1})}\lesssim_\varepsilon R^{\varepsilon}\|f\|_{L^n(\mathbb{R}^n)},
\end{equation}
stated here in an equivalent scaled form,
would imply the endpoint extension inequality \eqref{restconjend} via \eqref{Stconj} by an elementary duality argument.
We refer to \cite{HRZ,Zahl} and the references there for further discussion of the Kakeya conjecture, which is fully resolved only when $n=2$. We stress however that Stein's conjecture, with a Kakeya-type maximal operator $\mathcal{M}$ as above, is not satisfactorily resolved even for $n=2$, unless the weight $w$ has some very specific structure -- see, in particular, \cite{BRV}, \cite{CS}, \cite{BCSV} and \cite{BBC3}.

Our next theorem provides a variant of Stein's conjecture for the extension operator in the case $n=2$. Its statement naturally involves a bilinear analogue of the linear operator $T_\delta$ appearing in Section 1.2. It will be convenient to define this initially in abstract terms, as this sort of bilinearisation will also arise in the context of the spherical Radon transform $A_t$, also defined in Section 1.2. For an operator $T$, mapping functions on $\mathbb{S}^{n-1}$ to functions on $\mathbb{S}^{n-1}$, we define its bilinearisation $BT$ by the formula
\begin{equation}\label{e:Bilinearize}
BT(g_1,g_2)(\omega) = T(g_1(\cdot)\widetilde{g}_2(R_\omega(\cdot))(\omega),\;\;\; \omega\in\mathbb{S}^{n-1},
\end{equation}
where $\widetilde{g}(\omega)=\overline{g(-\omega)}$ and $R_\omega(\xi) = \xi - 2(\xi\cdot \omega) \omega$ is the reflection of $\xi$ in the hyperplane $\langle\omega\rangle^\perp$. \footnote{It is instructive to observe that if $H$ is the Hilbert transform on $\mathbb{S}^1$ then $BH$ becomes the classical bilinear Hilbert transform on $\mathbb{S}^1$; see \cite{LaceyThile}.}
In particular we have
\begin{equation}\label{e:BilinearT}
BT_\delta(g_1,g_2)(\omega) = \int_{\mathbb{S}^{n-1}} \frac{g_1(\xi) \widetilde{g}_2(R_\omega(\xi))}{|\omega\cdot \xi| +\delta}\, d\sigma(\xi).
\end{equation}
\begin{theorem}\label{t:wStein}
Let $n=2$. Then for all $R\gg1$,
\begin{align}\label{e:wStein0}
\int_{B_R} |\widehat{gd\sigma}|^2 w \lesssim & \int_{\mathbb{S}^1} BT_{1/R} \bigl( |g|_{1/R}^2, |g|_{1/R}^2 \bigr)(\omega)^\frac12 \mathcal{S}w(\omega)\, d\sigma(\omega) \\
& + \int_{\mathbb{S}^1}  BT_{1/R} \bigl( |g|_{1/R}^2, |g|_{1/R}^2 \bigr)(\omega^\perp)^\frac12 \mathcal{S}w(\omega)\, d\sigma(\omega), \nonumber
\end{align}
where
$$
\mathcal{S}w(\omega):= \left( \int_{\langle \omega\rangle^\perp} \bigl|(-\Delta_v)^\frac14 Xw(\omega,v)\bigr|^2\, d\lambda_\omega(v) \right)^\frac12,
$$
and $|g|_{1/R}$ is a suitable mollification of $|g|$ at scale $1/R$, such as that given by convolution with the Poisson kernel on $\mathbb{S}^1$.
\end{theorem}
As we shall see in Section \ref{Sect:appl}, the auxiliary bilinear operator $BT_{\delta}$ in Theorem \ref{t:wStein} is very well behaved, satisfying the bounds
\begin{align}
\big\| BT_{\delta} (g_1,g_2) \big\|_{L^\frac12(\mathbb{S}^1)} &\lesssim \log(\delta^{-1})^2 \|g_1\|_{L^1(\mathbb{S}^1)} \|g_2\|_{L^1(\mathbb{S}^1)}, \label{e:BT1/2} \\
\big\| BT_{\delta} (g_1,g_2) \big\|_{L^1(\mathbb{S}^1)} &\lesssim \log(\delta^{-1}) \|g_1\|_{L^2(\mathbb{S}^1)} \|g_2\|_{L^2(\mathbb{S}^1)} \label{e:BT1}.
\end{align}
In particular, our Stein-type inequality \eqref{e:wStein0}, when combined with \eqref{e:BT1/2}, immediately implies our Mizohata--Takeuchi-type inequality \eqref{e:wMizTak22}.
Similarly, \eqref{e:wStein0} and \eqref{e:BT1}, combined with the Cauchy--Schwarz inequality, imply the $n=2$ endpoint restriction inequality \eqref{restconjend}, thanks to the fact that the controlling operator $\mathcal{S}$, like $\mathcal{M}_R$, satisfies suitable bounds on $L^2(\mathbb{R}^2)$ -- indeed $\mathcal{S}$ is better behaved than $\mathcal{M}_R$ in this regard as $\|\mathcal{S}w\|_{L^2(\mathbb{S}^1)}\equiv \sqrt{2\pi}\|w\|_{L^2(\mathbb{R}^2)}$.

The operators $\mathcal{S}$ and $\mathcal{M}$ have notable similarities and differences. They are of course related via the numerology of the classical Hilbert--Sobolev embedding into $L^\infty(\mathbb{R})$. They differ in that $\mathcal{M}$ is monotone, while $\mathcal{S}$, which involves derivatives, is not. It is interesting to compare Theorem \ref{t:wStein}, and our approach to it, with the results and methods of Carbery and Seeger in \cite{CSe}.


In higher dimensions our results in the direction of Theorem \ref{t:wMizTak} are more complicated, although nonetheless they do constitute an improvement over what may be obtained from assuming the restriction conjecture -- see the forthcoming Theorem \ref{t:wMizTakHigh} for details.



It may be interesting to observe that our exploratory results from Sections \ref{sub2} and \ref{sub3} also bear some relation to the conjectures of Stein and Mizohata--Takeuchi. First of all, Theorem \ref{planch}, and the self-adjointness of $T_0$ 
imply that 
$$
\int_{\mathbb{R}^n} |\widehat{gd\sigma}|^2\, \mathcal{R}^*u = \int_{\mathbb{S}^{n-1}}|g|^2X_0\mathcal{R}^*u,
$$
provided the functions $g$ and $u$ satisfy a certain mutual support condition, ensuring finiteness of the expressions involved. This may be viewed as a certain improvement of \eqref{MizTak}, or indeed \eqref{Stconj}, for weights $w$ in the image of $\mathcal{R}^*$.

Theorems \ref{t:main} and \ref{t:n=3,r=infty} enjoy a rather different sort of interaction with \eqref{MizTak}.
In particular, \eqref{infX} allows one to deduce the Stein--Tomas restriction theorem
\begin{equation}
\|\widehat{gd\sigma}\|_{L^4(\mathbb{R}^3)}\lesssim\|g\|_{L^2(\mathbb{S}^2)}
\end{equation}
from \eqref{MizTak} -- one simply writes $|\widehat{gd\sigma}|^4=|\widehat{gd\sigma}|^2w$ where $w=|\widehat{gd\sigma}|^2$. Similarly, when $n=2$, Theorem \ref{t:main} shows that the endpoint restriction inequality \eqref{restconjend} follows from  \eqref{MizTak}.
We remark in passing that this very direct link between \eqref{MizTak} and the Stein--Tomas restriction theorem fails for $n>3$ as the inequality
\begin{equation}\label{e:4/n-1}
\big\| X( |\widehat{gd\sigma}|^{\frac{4}{n-1}}) \|_{L^\infty} \lesssim \| g \|_{L^2(\mathbb{S}^{n-1})}^{\frac{4}{n-1}}
\end{equation}
ceases to hold when the exponent $\tfrac{4}{n-1}$ falls below $2$.

\subsection*{Contextual remarks}
This paper emerged from an interest in further exploring ways in which $L^2$ methods might be applied to Fourier restriction theory. There have been a number of important attempts to ``reformulate" questions in restriction theory with either the input $g$ or output $\widehat{gd\sigma}$ belonging to a space with quadratic characteristics. An early example is the two-dimensional reverse Littlewood--Paley inequality of C\'ordoba and C. Fefferman (see \cite{Co}, \cite{FeffIsr}), which is closely-related to the $\ell^2$-decoupling inequalities (also known as Wolff inequalities) developed recently by Bourgain, Demeter and others \cite{BD}.
These inequalities are intimately related to the multilinear restriction theory, where the analogues of the endpoint estimate \eqref{restconjend} are often on $L^2$ -- see for example \cite{BCT}, \cite{BEsc}, \cite{BI} and \cite{Guth2}.  Perhaps the most natural setting is that of weighted $L^2$ spaces as discussed above -- see for example \cite{BRV}, \cite{CS}, \cite{BCSV} and \cite{DGOWWZ}. As we have already mentioned, particular inspiration for our work is that of Planchon and Vega \cite{PV} -- see also \cite{VegaEsc}, \cite{BBFGI}, \cite{BelVega}. There are points of contacts with other works, such as \cite{BS}, where $L^p$ averages of the extension operator over spheres are studied, or \cite{BGT}, where restrictions of eigenfunctions of the Laplace--Beltrami operator to submanifolds are considered.

Finally, we remark that one might sensibly expect some of the above considerations to generalise to $T_{k,n}(|\widehat{gd\sigma}|^2)$, where $T_{k,n}$ denotes the $k$-plane transform in $\mathbb{R}^n$ with $1<k<n-1$ -- indeed there are some related results of this type already in \cite{BelVega}.
In particular there is an evident (conjectural) $k$-plane generalisation of \eqref{e:EndRadon} and \eqref{e:2/n-1} whose statement we leave to the reader.
\subsubsection*{Organisation} The proofs of the theorems and propositions stated above will be presented in the following sections in the order that they appear.
\subsubsection*{Acknowledgments} We thank David Beltran, Neal Bez, Tony Carbery, Taryn Flock, Susana Guti\'errez, Marina Iliopoulou and Sanghyuk Lee for numerous stimulating discussions during the course of this work.


\section{Proof of Theorem \ref{planch}}
We begin by establishing the elementary identity \eqref{id1}, for which we may suppose that $\omega=e_n$ by rotation-invariance. By the support condition on $g$ we may write
$$\widehat{gd\sigma}(x',t) = \mathcal{F}_{n-1} \left[1_{B(0,1)}g\left(\cdot,\sqrt{1-|\cdot|^2}\right)e^{it\sqrt{1-|\cdot|^2}}(1-|\cdot|^2)^{-\frac12}\right](x'),$$ where $\mathcal{F}_{n-1}$ denotes the $(n-1)$-dimensional Fourier transform. Consequently, by Plancherel's theorem,
\begin{align*}
\mathcal{R}(|\widehat{gd\sigma}|^2)(e_n,t) & = \int_{\mathbb{R}^{n-1}} |\widehat{gd\sigma}(x',t)|^2\, dx' = \int_{\mathbb{S}^{n-1}} \frac{|g(\xi)|^2}{|\xi_n|}\, d\sigma(\xi)=T_0(|g|^2)(e_n).
\end{align*}
It remains to
show that $\mathcal{R}(X_0^*|g|^2)(e_n,t) = T_0(|g|^2)(e_n)$ for all $t\neq 0$.
Using \eqref{exact}, polar coordinates, and the support condition on $g$,
\begin{align*}
\mathcal{R}X^*_0(|g|^2)(e_n, t)
&=
\int_{\mathbb{S}^{n-1}} \int_{0}^\infty \delta(r\xi\cdot e_n -t) (|g(\xi)|^2 + |g(-\xi)|^2)\, drd\sigma(\xi)\\
&=
\1_{t>0} \int_{ \{\xi_n >0\}} \frac{|g(\xi)|^2}{|\xi_n|}\, d\sigma(\xi) + \1_{t<0} \int_{ \{\xi_n < 0\}} \frac{|g(-\xi)|^2}{|\xi_n|}\, d\sigma(\xi)\\
&=
T_0(|g|^2)(e_n)
\end{align*}
whenever $t\not=0$.

We now turn to \eqref{id3}, which will follow by similar reasoning to the above, combined with a routine mollification argument. Again, by rotation-invariance,  we may assume that $\omega=e_n$. Suppose that $\Psi\in C_c^{\infty}(\mathbb{R}^{n-1})$ and $\psi\in C_c^\infty(\mathbb{R})$ are such that both $\widehat{\Psi}$ and $\widehat{\psi}$ are equal to $1$ on the unit balls of their respective domains. Next we set $\Phi(\xi)=\Psi(\xi')\psi(\xi_n)$ and $\Phi_R(\xi)=\Psi_R(\xi')\psi_R(\xi_n)$, where $\Psi_R(\xi')=R^{n-1}\Psi(R\xi')$ and $\psi_R(\xi_n)=R\psi(R\xi_n)$ for each $R>0$. Of course, $$\int_{\mathbb{R}}|\Phi_R(\xi',\xi_n)|\, d\xi_n = \|\psi\|_1|\Psi_R(\xi')|,$$ and so by Plancherel's theorem,
\begin{align*}
\mathcal{R}(\1_R|\widehat{gd\sigma}|^2)(e_n,t)
&\le
\int_{\mathbb{R}^{n-1}} \big|{}^{\wedge}[\Phi_R \ast (gd\sigma)](x',t) \big|^2\, dx' \\
&\lesssim
\int_{\mathbb{R}^{n-1}} \Biggl(\int_{\mathbb{S}^{n-1}} |\Psi_R(\xi'-\eta')| |g(\eta)|\, d\sigma(\eta) \Biggr)^2\, d\xi'.
\end{align*}
By symmetry, it will suffice to bound the above expression with $g$ supported in the upper hemisphere $\mathbb{S}^{n-1}_+$. By the Cauchy--Schwarz inequality,
\begin{align*}
\Biggl( \int_{\mathbb{S}^{n-1}_+} |\Psi_R(\xi'-\eta')||g(\eta)|\, d\sigma(\eta) \Biggr)^2
=&
\Biggl(\int_{\mathbb{R}^{n-1}} |\Psi_R(\xi'-\eta')||g(\eta',\sqrt{1-|\eta'|^2})|\, \frac{d\eta'}{(1-|\eta'|^2)^{\frac12}} \Bigg)^2\\
\le&
|\Psi_R|\ast \big[ (1-|\cdot|^2)^{-\frac12} \big] (\xi') \cdot |\Psi_R|\ast \big[ \frac{|g(\cdot,\sqrt{1-|\cdot|^2})|^2}{(1-|\cdot|^2)^{\frac12}} \big](\xi'),
\end{align*}
and hence using the Fubini's theorem,
\[
\mathcal{R}(\1_{B_R}|\widehat{gd\sigma}|^2)(e_n,t) \lesssim \int_{\mathbb{R}^{n-1}} |\Psi_R| \ast  |\Psi_R|\ast \big[ (1-|\cdot|^2)^{-\frac12} \big] (\xi') \frac{|g(\xi',\sqrt{1-|\xi'|^2})|^2}{(1-|\xi'|^2)^{\frac12}}\, d\xi'.
\]
Since $|\Psi_R(\eta')| \lesssim_N R^{n-1} (1+|R\eta'|)^{-N}$ for any $N\in\mathbb{N}$, an elementary computation reveals that
\[
|\Psi_R| \ast  |\Psi_R|\ast \big[ (1-|\cdot|^2)^{-\frac12} \big] (\xi') \lesssim ((1-|\xi'|^2)^{\frac12}+ R^{-1})^{-1},
\]
which establishes \eqref{id3}.

\section{Proof of Theorem \ref{t:main}}
By Theorem \ref{planch}, it suffices to prove the following:
\begin{proposition}\label{mainprop}
If $1\leq p,q\leq\infty$ and
\begin{equation}\label{expoMike}
\frac1{p} \le \frac{n-1}{q},\ \frac1{q'} \le \frac{n-1}{p'},
\end{equation}
then
\begin{equation}\label{e:Mike}
\big\| T_\delta f \big\|_{L^q(\mathbb{S}^{n-1})} \lesssim \log\left(\frac{1}{\delta}\right) \| f \|_{L^p(\mathbb{S}^{n-1})}
\end{equation}
for all $\delta>0$.
\end{proposition}
We shall prove this proposition by reducing it to an endpoint bound on the operator $A_t$ defined in \eqref{MCopt}.
First of all, the exponents $p$ and $q$ may be increased and decreased, respectively in \eqref{e:Mike} by H\"older's inequality. The case $p=q$ follows immediately from the elementary inequality
\[
\sup_{x\in\mathbb{S}^{n-1}}\int_{\mathbb{S}^{n-1}}\, \frac{d\sigma(y)}{|x\cdot y|+\delta} \lesssim \log\left(\frac{1}{\delta}\right),
\]
and so, by interpolation, it is enough to prove \eqref{e:Mike} when $p'=q=n$.

Of course,
$$
T_\delta f(\omega)=\int_{-1}^1A_tf(\omega)\frac{dt}{|t|+\delta},
$$
and so it will be enough to prove the following:
\begin{lemma}\label{Chrot}
\begin{equation}\label{nuff}
\|A_t f\|_{L^n(\mathbb{S}^{n-1})} \lesssim \|f\|_{L^{\frac{n}{n-1}}(\mathbb{S}^{n-1})}
\end{equation}
uniformly in $t$ sufficiently small.
\end{lemma}
When $n=3$, Lemma \ref{Chrot} was obtained in \cite[Section 6]{Christ84}, and moreover it was shown that
\[
\| A_t f \|_{3} \le C (1-t^2)^{-\frac13} \| f \|_{\frac32}, \quad t\in (-1,1).
\]
In general dimensions, \eqref{nuff} was established for $t=0$, also in \cite{Christ84}, and so we only need to observe a proof of \eqref{nuff} that is suitably stable under perturbations of $t$ about $0$.
For this we appeal to the well-known theory of Radon-like transforms satisfying a rotational curvature condition. In order to state an appropriate result in this context, we let $\psi$ be a compactly-supported cut-off function on $\mathbb{R}^{n-1}\times \mathbb{R}^{n-1}$ and suppose that $\Phi\in C^\infty(\mathbb{R}^{n-1}\times \mathbb{R}^{n-1})$ satisfies \begin{equation}\label{e:rotcurve}
\mbox{rotcurv}(\Phi)(x,y):=\left| {\rm det}
\begin{pmatrix}
\Phi & \partial_{x_1}\Phi & \cdots & \partial_{x_{n-1}}\Phi \\
\partial_{y_1} \Phi & \partial^2_{x_1,y_1} \Phi & \cdots & \partial^2_{x_{n-1},y_1}\Phi \\
\vdots & \vdots & \ddots
& \vdots \\
\partial_{y_{n-1}}\Phi & \partial^2_{x_1,y_{n-1}}\Phi & \cdots & \partial^2_{x_{n-1},y_{n-1}}\Phi
\end{pmatrix}
\right| \ge \varepsilon>0,
\end{equation}
for all $(x,y) \in \{\Phi(x,y)=0\} \cap {\rm supp}(\psi)$.
\begin{lemma}\label{l:taovargasvega}
If $\Phi$ satisfies the rotational curvature condition \eqref{e:rotcurve} on the support of $\psi$, then the averaging operator
\[
Sf(x) = \int_{\mathbb{R}^{n-1}} f(y)\psi(x,y)\delta(\Phi(x,y))\, dy
\]
satsfies $\| Sf \|_{L^n(\mathbb{R}^{n-1})} \le C(\varepsilon) \| f \|_{L^{\frac{n}{n-1}}(\mathbb{R}^{n-1})}$.
\end{lemma}
We refer to \cite{TVV} for a short proof of this well-known result, and to \cite{BigStein} for further context.

We now turn to the proof of Lemma \ref{Chrot}.
We begin by fixing a parameter $0<\eta\ll 1$, which will be taken sufficiently small (depending on at most $n$), and let $\{U_\alpha\}$ be a cover of $\mathbb{S}^{n-1}$ by spherical caps
$$
U_\alpha=B(\alpha;\eta)\cap\mathbb{S}^{n-1},
$$
indexed by a (maximal) $O(\eta)$-separated set of points $\{\alpha\}$ on $\mathbb{S}^{n-1}$. Restricting attention to $|t|\leq\eta$, and writing
$
A_t^{\alpha,\beta}f=1_{U_\beta}A_t(1_{U_\alpha}f),
$
we have
$$
A_tf\leq\sum_{\alpha,\beta}A_t^{\alpha,\beta}f.
$$
Using the support property of the distributional kernel of $A_t$, we have that $A_t^{\alpha,\beta}=0$ if $\alpha\cdot\beta\gtrsim \eta$, and so it suffices to show that
\begin{equation}\label{alphanuff}
\|A_t^{\alpha,\beta}f\|_{L^n(\mathbb{S}^{n-1})} \lesssim \|f\|_{L^{\frac{n}{n-1}}(\mathbb{S}^{n-1})}
\end{equation}
whenever $\alpha\cdot\beta\lesssim \eta$. By enlarging $U_\alpha$ and $U_\beta$ by a constant factor (depending on at most $n$), and enlarging $\eta$ by a suitable constant factor, we may reduce to proving \eqref{alphanuff} in the case $\alpha\cdot\beta=0$. By rotation-invariance, we may further suppose that $\alpha$ and $\beta$ are the standard basis vectors $e_{n-1}$ and $e_{n-2}$ respectively. Parametrising $U_\alpha$ and $U_\beta$ in the natural way, namely via the mappings
$$
B(0;\eta)\ni y\mapsto (y_1,\hdots,y_{n-1},\sqrt{1-|y|^2})
$$
and
$$
B(0;\eta)\ni x\mapsto (x_1,\hdots,x_{n-2},\sqrt{1-|x|^2}, x_{n-1}),$$
it suffices to prove that for $\eta$ sufficiently small,
\begin{equation}\label{rotnuff}
\mbox{rotcurv}(\Phi_t)(x,y)\geq\frac{1}{2}
\end{equation}
on $B(0;\eta)\times B(0;\eta)$, uniformly in $|t|\leq \eta$,
where $\Phi_t(x,y)=\Phi_0(x,y)-t$, and
$$
\Phi_0(x,y)=\sum_{j=1}^{n-2}x_jy_j+x_{n-1}\sqrt{1-|y|^2}+y_{n-1}\sqrt{1-|x|^2}.
$$
An elementary calculation now reveals that $\mbox{rotcurv}(\Phi_0)(0,0)=1$, and so provided $\eta$ is taken sufficiently small (depending only on $n$), the inequality \eqref{rotnuff} follows for $|t|\leq\eta$
for sufficiently small $\eta$ by the smoothness of $\Phi$.

We end this section by showing that the range of exponents \eqref{expoRad} in Theorem \ref{t:main} is best-possible using Knapp-type examples related to those in \cite{TVV}.
In view of \eqref{statphase}, the necessity of $r=\infty$ follows quickly by applying \eqref{e:radon-extension} with $g\equiv 1$.
To obtain the other conditions we define
\begin{equation}\label{e:ex}
g_m(\xi) = \1_{\{ |(\xi_1,\ldots,\xi_m)|\leq \delta \}},\;\;\;S_m = \{ x\in \mathbb{R}^n: |(x_1,\ldots,x_{n-m})| \leq \delta^{-1}, |(x_{n-m+1},\ldots, x_n)| \leq \delta^{-2} \}
\end{equation}
for each $\delta>0$ and $1\le m\le n$.
A standard stationary phase argument reveals that
\begin{equation}\label{e:StationaryPhase}
|\widehat{g_md\sigma}(x)|^2 \gtrsim \delta^{2(n-1)}
\end{equation}
on a large portion of $S_m$.
On the other hand, elementary geometric considerations reveal that
\begin{align*}
\big\| \mathcal{R}(\1_{S_m}) \big\|_{L^q_\theta L^\infty_v}
\gtrsim
\max\{     \delta^{-n-m+2} ,
    \delta^{-(n-1+m)}  \delta^{\frac{m}{q}}
    \}
\end{align*}
for all $1\le m \le n$. The first of these lower bounds takes account of only tangential interactions between $S_m$ and the $(n-1)$-planes, while the second takes account of only transversal interactions. Applying this lower bound to \eqref{e:radon-extension}, along with \eqref{e:StationaryPhase} and the fact that $\|g_m\|_p^2 \sim \delta^{2\frac{n-m}p}$, we obtain the necessary conditions
\begin{align*}
\frac1p &\le \frac12, \;\;\; \frac{2(n-m)}p \le (n-1) -m + \frac{m}q,
\end{align*}
for all $1\le m \le n$. The conditions \eqref{expoRad} now follow by considering $m=1$ and $m=n-1$ here.

\section{Proof of Proposition \ref{prop:implication}}\label{AP1}
We begin by observing that \eqref{e:2/n-1} is equivalent to
\begin{equation}\label{e:rescale}
\big\| X( \1_{B_1} |\widehat{gd\sigma}(R\cdot)|^{\frac{2}{n-1}} ) \big\|_{L^n_\omega L^\infty_v} \lesssim_\varepsilon R^\varepsilon R^{-1} \|g\|_{L^{\frac{2n}{n-1}}(\mathbb{S}^{n-1})}^{\frac{2}{n-1}}
\end{equation}
by scaling.
As we clarify next, uncertainty principle considerations essentially allow one to replace the integration along line segments in \eqref{e:rescale} by averaging on $\frac1R$-neighbourhoods of line segments, whereby the statement \eqref{e:rescale} may be rephrased in terms of the classical Kakeya maximal function. While this may be expected, the details of this reduction are not altogether routine for $n>3$. The distinction arises due to the fact that $\frac2{n-1}<1$ when $n>3$, and the subsequent inapplicability of Minkowski's inequality. A similar issue arises in the context of multilinear restriction estimates, which typically have Lebesgue exponents below $1$ -- see \cite{TaoFlow} for further discussion.

For $\delta>0$, a locally integrable function $f:\mathbb{R}^n\rightarrow\mathbb{C}$ and $\omega\in\mathbb{S}^{n-1}$ we define the Kakeya maximal function $\mathcal{K}_{\delta}$ by
$$
\mathcal{K}_\delta f(\omega)=\sup_{T\parallel\omega}\frac{1}{|T|}\int_T|f|,
$$
where the supremum is taken over all $\delta$-tubes parallel to the direction $\omega$. Here, as usual, a $\delta$-tube is $\delta$-neighbourhood of a unit line segment, and its direction is that of its central line. As is well-known \cite{BCSS}, the restriction conjecture \eqref{restconjend} implies the estimate
\begin{equation}\label{kakconj}
\|\mathcal{K}_\delta f\|_{L^n(\mathbb{S}^{n-1})}\lesssim \delta^{-\varepsilon}\|f\|_{L^n(\mathbb{R}^n)},
\end{equation}
which is referred to as the Kakeya maximal conjecture; we refer back to \eqref{kakeyaconjecturescaled} for an equivalent statement where the tubes are scaled to have unit width. The following lemma states, to all intents and purposes, that one may replace $X$ by $\mathcal{K}_{1/R}$ in \eqref{e:rescale}.
\begin{lemma}\label{l:KakeXray}
Let $n\ge2$ and $P_t$ be the Poisson kernel for $t>0$:
$$
P_t(x) = \frac{\Gamma [ \frac{n+1}2 ]}{\pi^\frac{n+1}2} \frac{t}{(t^2+|x|^2)^\frac{n+1}2}.
$$
\begin{enumerate}
\item\label{i1}
The inequality \eqref{e:rescale} implies
\begin{equation}\label{e:rescalekak}
\big\| \mathcal{K}_{1/R}( \1_{B_1} |\widehat{gd\sigma}(R\cdot)|^{\frac{2}{n-1}} ) \big\|_{L^n(\mathbb{S}^{n-1})} \lesssim_\varepsilon R^\varepsilon R^{-1} \|g\|_{L^{\frac{2n}{n-1}}(\mathbb{S}^{n-1})}^{\frac{2}{n-1}}.
\end{equation}
\item\label{i2}
Conversely, \eqref{e:rescale} follows from
\begin{equation}\label{e:Poisson}
\big\| \mathcal{K}_{1/R}\big( \1_{B_1} (P_{1/R}\ast |\widehat{gd\sigma}(R\cdot)| )^{\frac{2}{n-1}} \big) \big\|_{L^n(\mathbb{S}^{n-1})} \lesssim_\varepsilon R^\varepsilon R^{-1} \|g\|_{L^{\frac{2n}{n-1}}(\mathbb{S}^{n-1})}^{\frac{2}{n-1}}.
\end{equation}
\end{enumerate}
\end{lemma}
\begin{remark}
As we have already indicated, when $n=2, 3$ the statement of Lemma \ref{l:KakeXray} follows by an entirely standard mollification argument -- indeed it is straightforward to see that \eqref{e:rescale} and \eqref{e:rescalekak} are equivalent in those cases.  As we shall see, the presence of the Poisson kernel $P_t$ above stems from the convenient fact that it is essentially constant, or comparable to itself, at scale $O(t)$; that is, $P_t(y)\lesssim P_t(x)$ whenever $|x-y|\lesssim t$, for suitably chosen implicit constants.
\end{remark}
\begin{proof}
Part {\it (1)} is a direct consequence of the pointwise inequality $\mathcal{K}_{1/R}f(\omega)\lesssim \sup_v Xf(\omega,v)$, which holds for any nonnegative $f$ with support in the unit ball, and so we focus on Part {\it (2)}. Since $P_t$ is the Fourier transform of $e^{-2\pi |\cdot|}$ we have that
$$
\widehat{gd\sigma}=e^{2\pi}P_1\ast \widehat{gd\sigma},
$$
so that after scaling it follows that
$$
|\widehat{gd\sigma}(Rx)| \lesssim P_{1/R}\ast |\widehat{gd\sigma}(R\cdot)|(x).
$$
Since $P_{1/R}$ is comparable to itself at scale $O(R^{-1})$ as discussed above, we may conclude that
$$
|\widehat{gd\sigma}(Rx)| \lesssim P_{1/R}\ast |\widehat{gd\sigma}(R\cdot)|(x')
$$
whenever $|x-x'|\lesssim \frac1R$. Hence
$$
X(\1_{B_1} |\widehat{gd\sigma}(R\cdot)|^\frac2{n-1})(\omega,v)
\lesssim
X(\1_{B_1} (P_{1/R}\ast |\widehat{gd\sigma}(R\cdot)| )^\frac2{n-1})(\omega,v')
$$
whenever $|v-v'|\lesssim \frac1R$,
and so by averaging in such $v'$ we obtain
$$
X(\1_{B_1} |\widehat{gd\sigma}(R\cdot)|^\frac2{n-1})(\omega,v)
\lesssim
\mathcal{K}_{1/R}(\1_{B_1} (P_{1/R}\ast |\widehat{gd\sigma}(R\cdot)| )^\frac2{n-1})(\omega)
$$
uniformly in $v$.
\end{proof}
We now turn to the proof of Proposition \ref{prop:implication}, beginning with the assertion that the restriction conjecture implies Conjecture \ref{conj:Xrest}.
By Lemma \ref{l:KakeXray} this may be reduced to showing that \eqref{restconjend} implies \eqref{e:Poisson}. Since \eqref{restconjend} implies \eqref{kakconj}, we have
\begin{align*}
\big\| \mathcal{K}_{1/R}\big( \1_{B_1} (P_{1/R}\ast |\widehat{gd\sigma}(R\cdot)| )^{\frac{2}{n-1}} \big) \big\|_{L^n(\mathbb{S}^{n-1})}
&\lesssim_\varepsilon R^{\varepsilon}  \big\| \1_{B_1} (P_{1/R}\ast |\widehat{gd\sigma}(R\cdot)| ) \big\|_{L^\frac{2n}{n-1}(\mathbb{R}^{n})}^{\frac{2}{n-1}} \\
&\le R^{\varepsilon} \big( {\rm I}_1 +  \sum_{j\ge2} {\rm I}_j  \big)^\frac2{n-1},
\end{align*}
where
$$
{\rm I}_1 = \big\| P_{1/R}\ast ( \1_{B_2}|\widehat{gd\sigma}(R\cdot)|)  \big\|_{L^\frac{2n}{n-1}(B_1)}
$$
and
$$
{\rm I}_j = \big\| P_{1/R}\ast ( \1_{|\cdot|\sim 2^j}|\widehat{gd\sigma}(R\cdot)|) \big\|_{L^\frac{2n}{n-1}(B_1)}
$$
for $j\geq 2$.
For the first term, after using Minkowski's inequality to remove the fixed averaging operator $P_{1/R}\ast$, we have
$$
{\rm I}_1 \lesssim \big\|  \1_{B_2} \widehat{gd\sigma}(R\cdot) \big\|_{L^\frac{2n}{n-1}(\mathbb{R}^{n})} \lesssim_\varepsilon R^\varepsilon R^{-\frac{n-1}2} \|g\|_{L^\frac{2n}{n-1}(\mathbb{S}^{n-1})}
$$
by a further application of \eqref{restconjend}.
The estimates for the remainder terms ${\rm I}_j$ are similar and summable in $j$ thanks to the decay of $P_t$. Specifically, for each $x\in B_1$,
\begin{align*}
P_{1/R}\ast ( \1_{|\cdot|\sim 2^j}|\widehat{gd\sigma}(R\cdot)| )(x)
&\sim
2^{-j(n+1)} R^{-1} \int_{|y|\sim 2^j} |\widehat{gd\sigma}(Ry)|\, dy \\
&\lesssim
2^{-j(n+1)} R^{-1} (2^{jn})^{\frac{n+1}{2n}}  \big( \int_{|y|\sim 2^j} |\widehat{gd\sigma}(Ry)|^\frac{2n}{n-1}\, dy \big)^\frac{n-1}{2n} \\
&= (2^j R)^{-\frac{n+1}2} \| \1_{B_{2^jR}} \widehat{gd\sigma} \|_{L^\frac{2n}{n-1}}.
\end{align*}
Applying \eqref{restconjend} we obtain
$$
{\rm I}_j \le \big\| P_{1/R}\ast ( \1_{|\cdot|\sim 2^j}|\widehat{gd\sigma}(R\cdot)|) \big\|_{L^\infty(B_1)} \lesssim_\varepsilon (2^jR)^{-\frac{n+1}2 +\varepsilon} \|g\|_{L^\frac{2n}{n-1}(\mathbb{S}^{n-1})} \le (2^jR)^{-\frac{n-1}2} \|g\|_{L^\frac{2n}{n-1}(\mathbb{S}^{n-1})}.
$$
 The inequality \eqref{e:Poisson} now follows by combining the above estimates and summing in $j$.

To complete the proof of Proposition \ref{prop:implication} it remains to show that \eqref{e:rescale} implies \eqref{kakconj}.
It is well known that \eqref{kakconj} has an equivalent dual form which states that
\begin{equation}\label{e:kakeya}
\big\| \sum_{T \in \mathbb{T}} \1_T \big\|_{L^{\frac{n}{n-1}}(\mathbb{R}^n)} \lesssim_\varepsilon R^\varepsilon \big(R^{-\frac{n-1}2} \#\mathbb{T} \big)^{\frac{n-1}{n}}
\end{equation}
holds true for all families $\mathbb{T}$ of $R^{-\frac12}$-tubes contained in an $O(1)$ ball whose directions form a $R^{-\frac12}$-separated subset of $\mathbb{S}^{n-1}$; see \cite{TaoNotes} for instance.
So it suffices to show \eqref{e:kakeya} assuming \eqref{e:rescale}.
We begin by establishing the natural Kakeya-type consequence of \eqref{e:rescale}, or equivalently \eqref{e:2/n-1}, which follows by a routine randomisation argument.
\begin{lemma}\label{l:RestKake}
Suppose that \eqref{e:rescale} holds, the family of tubes $\mathbb{T}$ is as above and
$$
F:=(\sum_{T \in \mathbb{T}} \1_T)^{\frac1{n-1}}.
$$
Then
\begin{equation}\label{e:XrayKake}
\big\| XF \big\|_{L^n_\omega L^\infty_v} \lesssim_\varepsilon R^\varepsilon \big( R^{-\frac{n-1}{2} } \# \mathbb{T}\big)^\frac1n,
\end{equation}
where the implicit constant is independent of $\mathbb{T}$.
\end{lemma}
\begin{proof}
For a tube $T\in\mathbb{T}$, let $\omega(T) \in \mathbb{S}^{n-1}$ denote its direction, and let $\mathcal{C}(\omega) = B(\omega,cR^{-\frac12})\cap \mathbb{S}^{n-1}$ for some sufficiently small constant $c$ depending only on the dimension. Elementary considerations reveal that for suitably chosen modulations $e_T$, the functions $\phi_T:=1_{\mathcal{C}(\omega(T))}e_T$ satisfy
\begin{equation}\label{e:pointwise}
|\widehat{\phi_T d\sigma}(x)| \gtrsim R^{-\frac{n-1}2} \1_{T}(R^{-1}x),
\end{equation}
uniformly in $T\in\mathbb{T}$. Note also that the constant $c$ may be chosen small enough so that the caps $\{\mathcal{C}(\omega(T) \}_{T\in\mathbb{T}}$ are disjoint, since the directions of tubes are $R^{-\frac12}$-separated. Next we let $\nu=\{ \nu_T\}_{T\in\mathbb{T}}$ be a sequence of independent random variables taking values in $\{-1,1\}$ with equal probability, and
$$
g_{\nu}= \sum_{T\in\mathbb{T}} \nu_T \phi_T.
$$
Taking expectations,  using Khintchine's inequality and \eqref{e:pointwise}, we have
\begin{equation}\label{e:point2}
\mathbf{E}(|\widehat{g_\nu d\sigma}(Rx)|^\frac{2}{n-1})
\sim
\big( \sum_{T\in\mathbb{T}}  |\widehat{\phi_T d\sigma}(Rx)|^2  \big)^{\frac1{n-1}} \gtrsim R^{-1} \big( \sum_{T\in\mathbb{T}} \1_T (x)\big)^\frac1{n-1} = R^{-1}F(x).
\end{equation}
By the linearity of $X$, the inequality \eqref{e:point2}, Minkowski's inequality and \eqref{e:rescale},
we conclude that
\begin{align*}
\big\| XF \big\|_{L^{n}_\omega L^\infty_v}
&\lesssim
R\:\big\| \mathbf{E}(X(|\widehat{g_\nu d\sigma}(R\cdot)|^\frac2{n-1}))\big\|_{L^{n}_\omega L^\infty_v} \\
&\lesssim  R \:\mathbf{E}\bigl(\big\| X(\1_{B_1}|\widehat{g_\nu d\sigma}(R\cdot)|^\frac2{n-1})\big\|_{L^{n}_\omega L^\infty_v}\bigr)\\
&\lesssim_\varepsilon R^\varepsilon \:\mathbf{E}(\|g_\nu\|_{\frac{2n}{n-1}}^\frac2{n-1})
\sim
R^\varepsilon  \big( R^{-\frac{n-1}{2}}\#\mathbb{T} \big)^\frac1n,
\end{align*}
as required.
\end{proof}
In light of Lemma \ref{l:RestKake}  we have only to show the implication from \eqref{e:XrayKake} to \eqref{e:kakeya}.
Using \eqref{e:XrayKake} and the fact that $\mathcal{K}_{R^{-1/2}}F(\omega) \lesssim \sup_v XF(\omega,v)$, we have
\begin{equation}\label{e:KakeKake}
\big\| \mathcal{K}_{R^{-1/2}}F \big\|_{L^n_\omega L^\infty_v} \lesssim_\varepsilon R^\varepsilon \big( R^{-\frac{n-1}{2} } \# \mathbb{T}\big)^\frac1n.
\end{equation}
Since $\mathcal{K}_{R^{-1/2}}F$ is essentially constant at scale $R^{-\frac12}$, and $|\mathcal{C}(\omega(T))| \sim R^{-\frac{n-1}{2}}$,
\begin{align*}
\big\| \sum_{T \in \mathbb{T}} \1_T \big\|_{L^{\frac{n}{n-1}}(\mathbb{R}^n)}^\frac{n}{n-1}
&=
\int_{\mathbb{R}^n} \sum_{T\in\mathbb{T}}\1_T(x)F(x)\, dx \\
&\lesssim
R^{-\frac{n-1}{2}}  \sum_{T\in\mathbb{T}} \mathcal{K}_{R^{-1/2}}F(\omega(T)) \\
&\sim
\sum_{T\in\mathbb{T}} \int_{\mathcal{C}(\omega(T))} \mathcal{K}_{R^{-1/2}}F(\omega)\, d\sigma(\omega).
\end{align*}
Here the caps $\mathcal{C}(\omega)$ are as in the proof of Lemma \ref{l:RestKake}. Using H\"{o}lder's inequality, \eqref{e:KakeKake} and the fact that $$|\bigcup_{T\in\mathbb{T}} \mathcal{C}(\omega(T)) | \sim R^{-\frac{n-1}{2}}\# \mathbb{T},$$
we conclude that
\begin{align*}
\big\| \sum_{T \in \mathbb{T}} \1_T \big\|_{L^{\frac{n}{n-1}}(\mathbb{R}^n)}^\frac{n}{n-1}
&\lesssim
\big\| \mathcal{K}_{R^{-1/2}}F \big\|_{L^n(\mathbb{S}^{n-1})} |\bigcup_{T\in\mathbb{T}} \mathcal{C}(\omega(T)) |^\frac{n-1}{n} \\
&\lesssim_\varepsilon R^\varepsilon \big(R^{-\frac{n-1}{2}}\# \mathbb{T}\big)^\frac1n \big(R^{-\frac{n-1}{2}}\# \mathbb{T}\big)^\frac{n-1}{n},
\end{align*}
as claimed.

\section{Proof of Theorem \ref{planchX}}
%
By Fubini's theorem,
\begin{equation}\label{fub}
\int g(\xi)d\sigma(\xi)=\int_{-1}^1\int g(\xi)d\sigma_{\omega,t}(\xi)dt,
\end{equation}
Consequently,
$$
\widehat{gd\sigma}(x)=\int_{-1}^1\widehat{gd\sigma}_{\omega,t}(x)dt.
$$
Defining the projection $\pi_\omega$ by $\pi_\omega(\xi) = \xi -(\xi\cdot \omega) \omega$, we have
\begin{eqnarray*}
\begin{aligned}
\widehat{gd\sigma}_{\omega,t}(x)
&=\int g(\xi)e^{ix\cdot((\omega\cdot\xi)\omega+\pi_\omega(\xi))}d\sigma_{\omega,t}(\xi)
=e^{itx\cdot\omega}\widehat{gd\sigma}_{\omega,t}(\pi_\omega(x)).
\end{aligned}
\end{eqnarray*}
Hence
$$
\widehat{gd\sigma}(x)=\int_{-1}^1e^{itx\cdot\omega}\widehat{gd\sigma}_{\omega,t}(\pi_\omega(x))dt,
$$
and so,
\begin{eqnarray}\label{e:Xform}
\begin{aligned}
X(|\widehat{gd\sigma}|^2)(\omega,v)
&=\int_{\mathbb{R}}\left|\int_{-1}^1e^{ist}\widehat{gd\sigma}_{\omega,t}(v)dt\right|^2ds
=2\pi\int_{-1}^1|\widehat{gd\sigma}_{\omega,t}(v)|^2dt,
\end{aligned}
\end{eqnarray}
establishing \eqref{idX1}.
To establish \eqref{idX2} from \eqref{idX1} we first apply the elementary bound $$|\widehat{gd\sigma}_{\omega,t}(v)|\leq\int |g|d\sigma_{\omega,t}=A_t(|g|)(\omega),$$
which holds with equality if $g$ is single-signed, to obtain
$$
\sup_{v\in\langle\omega\rangle^\perp}X(|\widehat{gd\sigma}|^2)(\omega,v)\leq 2\pi S(|g|)(\omega)^2.
$$
Since $A_t(|g|)(\omega)=\widehat{|g|d\sigma}_{\omega,t}(0)$, it follows that
$$2\pi S(|g|)(\omega)^2=X_0(|\widehat{|g|d\sigma}|^2)(\omega),$$
by reversing the argument in \eqref{e:Xform}.


\section{Proof of Theorem \ref{t:n=3,r=infty}}\label{Sec6}
Although Theorem \ref{t:n=3,r=infty} makes reference to three dimensions only, much of our argument continues to function in all dimensions $n\geq 2$. 
In particular, we shall reduce Theorem \ref{t:n=3,r=infty} to certain weighted inequalities for the extension operator, which are potentially of independent value, and are naturally presented in any dimension. Consequently we shall work in general $n\geq 2$ dimensions much of the time.
Of course, Theorem \ref{t:main} establishes that no global estimates of the form \eqref{e:r=infty} are available when $n=2$, since $\mathcal{R}=X$ in that case. For $n>3$, our argument does yield estimates of the form \eqref{e:r=infty}, although it appears to fall short of providing a full characterisation of the admissible exponents.

As we shall see next, Theorem \ref{planchX} allows us to reduce the estimate \eqref{e:r=infty} to a weighted estimate for the extension operator. This is the content of our next lemma, which is phrased in terms of the classical Lorentz spaces.
\begin{lemma}\label{l:X-reduction}
For all $q\in[1,\infty)$,
\begin{equation}\label{e:0805-1}
\Big\| \sup_{v\in \langle \omega \rangle^\perp} X( |\widehat{gd\sigma}|^2) \Big\|_{L^q_\omega(\mathbb{S}^{n-1})} \lesssim \big\| \widehat{|g|d\sigma}(\cdot) |\cdot|^{\frac{1}{2} - \frac{n}{2q}} \big\|_{L^{2q,2}(\mathbb{R}^n)}^2.
\end{equation}
Moreover, equality holds in \eqref{e:0805-1} when $g$ is single-signed and $q=1$.
\end{lemma}
\begin{proof}
By \eqref{idX2},
\begin{equation}\label{e:0827-2}
\Bigl\|\sup_{v\in \langle \omega \rangle^\perp} X(|\widehat{gd\sigma}|^2)\Bigr\|_{L^q(\mathbb{S}^{n-1})} \le
\|X_0(|\widehat{|g|d\sigma}|^2)\|_{L^q(\mathbb{S}^{n-1})},
\end{equation}
with equality when $g$ is single-signed. Furthermore, \eqref{e:0827-2} allows us to assume $g\ge0$ for the remainder of our argument. Estimating further, we have
\begin{eqnarray*}
\begin{aligned}
\big\| X_0( |\widehat{g d\sigma}|^2)  \big\|_{L^{q}(\mathbb{S}^{n-1})}^{q}
= &
2\int_{\mathbb{R}^n} |\widehat{g d\sigma}(x)|^2 X_0( |\widehat{g d\sigma}|^2)(x/|x|)^{q-1}|x|^{-(n-1)}\, dx \\
\lesssim& \big\| |\widehat{ g d\sigma}|^2 |\cdot|^{-(n-1) + \frac{n}{q'}}  \big\|_{L^{q,1}(\mathbb{R}^n)} \big\| X_0( |\widehat{g d\sigma}|^2)(x/|x|)^{q-1}|x|^{-\frac{n}{q'}}  \big\|_{L_x^{q',\infty}(\mathbb{R}^n)} \\
=& \big\| \widehat{ g d\sigma} |\cdot|^{-\frac{n-1}{2} + \frac{n}{2q'}}  \big\|_{L^{2q,2}(\mathbb{R}^n)}^2 \big\| X_0( |\widehat{ g d\sigma}|^2) \big\|_{L^{q}(\mathbb{S}^{n-1})}^{q-1}.
\end{aligned}
\end{eqnarray*}
Here we have used the change of variables $(r,\omega)\in\mathbb{R}_{+} \times \mathbb{S}^{n-1} \mapsto x \in \mathbb{R}^n$, H\"{o}lder's inequality on Lorentz spaces and the fact that $|\cdot|^{-\frac{n}{q'}} \in L^{q',\infty}(\mathbb{R}^n)$.
Hence
\[
\big\| X_0( |\widehat{g d\sigma}|^2) \big\|_{L^{q}(\mathbb{S}^{n-1})} \le C \big\| \widehat{g d\sigma} |\cdot|^{-\frac{n-1}{2} + \frac{n}{2q'}}  \big\|_{L^{2q,2}(\mathbb{R}^n)}^2,
\]
which, together with \eqref{e:0827-2}, concludes the proof of \eqref{e:0805-1}.
Finally we note that every inequality in the above may be replaced by an equality when $q=1$ and $g$ is single-signed.
\end{proof}

As we shall clarify below, Lemma \ref{l:X-reduction} allows us to reduce Theorem \ref{t:n=3,r=infty} to the following proposition.
\begin{proposition}\label{sufprop} Let $X= (\frac23,\frac13), Y=(\frac34,\frac12), Z = (\frac34,1)$ and $\gamma=\gamma(p,q) = \frac1{q} -\frac1{p'}$. Then
\begin{equation}\label{e:0307}
\big\| \widehat{gd\sigma} \langle \cdot\rangle^{-\gamma} \big\|_{L^{2q,2}(\mathbb{R}^3)} \lesssim \|g\|_{L^{p}(\mathbb{S}^2)},
\end{equation}
for all $(\frac1p,\frac1q) \in (X,Y) \cup (Y,Z)$. Here we use $(A,B)$ to denote the open line segment between two points $A,B \in [0,1]^2$, and $\langle x\rangle=(1+|x|^2)^{\frac12}$.
\end{proposition}
Proposition \ref{sufprop} includes an endpoint case of some classical results of Bloom and Sampson \cite{BS}. We refer the reader to the appendix for further discussion and proofs of such statements in all dimensions.

It remains to deduce Theorem \ref{t:n=3,r=infty} from Proposition \ref{sufprop}.
We have only to establish \eqref{e:n=3r=infty} at the endpoint $(\frac12,0)$ and points in $(X,Y) \cup (Y,Z)$, since the remaining bounds follow from these by H\"{o}lder's inequality and interpolation.
We begin with the point $(\frac12,0)$, which follows quickly from Theorem \ref{planchX}, and indeed holds in all dimensions $n\geq 3$.
As is well-known, the distribution $d\sigma_{\omega,t}$ has total mass $\mathfrak{C}_n(1-t^2)^{\frac{n-3}{2}}$, where
$$
\mathfrak{C}_n = |\mathbb{S}^{n-3}| \int_0^1 \frac{t^{n-3}}{\sqrt{1-t^2}}\, dt.
$$
Hence by the Cauchy--Schwarz inequality,
\begin{equation}\label{margest}
|\widehat{gd\sigma}_{\omega,t}(v)|^2\leq\left(\int |g|d\sigma_{\omega,t}\right)^2\leq \mathfrak{C}_n (1-t^2)^{\frac{n-3}{2}}\int|g|^2d\sigma_{\omega,t}\leq \mathfrak{C}_n\int|g|^2d\sigma_{\omega,t},
\end{equation}
and so
$$
X(|\widehat{gd\sigma}|^2)(\omega,v)\leq 4\pi\mathfrak{C}_n\|g\|_2^2
$$
by \eqref{idX1}. We note that in the case $n=3$ we obtain the sharp inequality \eqref{infX} since $\mathfrak{C}_3 = \frac{\pi}{2}$.
Indeed, if we choose $g \equiv 1$ and $v=0 \in \langle \omega\rangle^\perp$, then we have from $\int_{\mathbb{S}^2}\, d\sigma(\xi) = 2\pi$ that
$$
\frac{X(|\widehat{d\sigma}|^2)(\omega,0)}{\|1\|_2^2} = 2\pi^2.
$$
It remains to deduce \eqref{e:n=3r=infty} at an arbitrary point $(\frac1p,\frac1q) \in (X,Y)\cup (Y,Z)$, beginning with the segment $(X,Y)$. By Lemma \ref{l:X-reduction} it suffices to establish
$$
\big\| \widehat{gd\sigma} |\cdot|^{-(\frac3{2q} - \frac12)}  \big\|_{L^{2q,2}(\mathbb{R}^3)} \lesssim \| g \|_{L^{p}}.
$$
Since the exponent $\gamma= \frac3{2q} - \frac12$ on $(X,Y)$, this is a consequence of Proposition \ref{sufprop} thanks to the elementary estimate
$$
\big\| \widehat{gd\sigma} |\cdot|^{-(\frac{3}{2q} - \frac12)}\1_{B(0,1)} \big\|_{L^{2q,2}(\mathbb{R}^3)} \leq \|\widehat{gd\sigma}\|_\infty\big\| |\cdot|^{-(\frac{3}{2q} - \frac12)}\1_{B(0,1)} \big\|_{L^{2q,2}(\mathbb{R}^3)}\lesssim \|g\|_{L^{p}(\mathbb{S}^2)}.$$
The reduction for the segment $(Y,Z)$ follows similarly.

We conclude this section by establishing the necessity of \eqref{e:n3rinfty-nec} in Theorem \ref{t:n=3,r=infty}. To this end, we employ the example \eqref{e:ex} with $m=1$. As before, we see from elementary geometric considerations that
\begin{align*}
\big\| X(\1_{S_1}) \big\|_{L^q_\theta L^\infty_v}
\gtrsim
    \max\{ \delta^{-1} ,
    \delta^{-2} \delta^{\frac{2}q}
    \}.
\end{align*}
In view of \eqref{e:StationaryPhase} and $\|g_1\|_p^2 \sim \delta^{\frac{4}p}$, \eqref{e:n=3r=infty} implies \eqref{e:n3rinfty-nec}, with the exception of the condition $(\frac1p,\frac1q) \neq (\frac34,\frac12)$. To see this, we need a more delicate lower bound on $ \| X(\1_{S_1}) \|_{L^2_\omega L^\infty_v}$ which takes into account contributions from transversal interactions at all scales. This reveals that
\begin{equation}\label{multiscalebound}
\| X(\1_{S_1}) \|_{L^2_\omega L^\infty_v} \gtrsim \log(\delta^{-1})^{\frac12} \delta^{-1},
\end{equation}
which of course forces $(\frac1p,\frac1q) \neq (\frac34,\frac12)$.

\section{X-ray estimates with applications to restriction theory}\label{Sect:appl}
Here we provide the proofs of Theorems \ref{t:wMizTak} and \ref{t:wStein}, and establish some analogous results in higher dimensions.
We begin with a simple geometrical observation, valid in all dimensions, and involving the operator $BA_t$, the bilinear version of $A_t$ given by \eqref{e:Bilinearize}. Explicitly, for nonnegative functions $g_1,g_2$ on $\mathbb{S}^{n-1}$, $t\in(-1,1)$ and $\omega \in \mathbb{S}^{n-1}$, we have
\begin{equation}\label{e:Leia}
BA_t(g_1,g_2)(\omega) = A_t(g_1(\cdot){\widetilde{g}_2(R_\omega\cdot)})(\omega) = \int_{\mathbb{S}^{n-1}} \delta(\omega\cdot \xi - t) g_1(\xi) {\widetilde{g_2}(R_\omega(\xi))}\, d\sigma(\xi). 
\end{equation}
This operator emerges naturally in this context since
\begin{equation}\label{e:conv}
(g_1d\sigma)\ast (g_2d\sigma)(x) = c_n \frac{\1_{|x|< 2}}{|x|} BA_{|x|/2}(g_1,g_2)(x/|x|).
\end{equation}
Of course \eqref{e:Leia} is rather special in the case $n=2$ since
\begin{align*}
BA_t(g_1,g_2)(\omega) =&\frac{g_1\bigl(t\omega + \sqrt{1-t^2} \omega^\perp\bigr)g_2\bigl(t\omega - \sqrt{1-t^2} \omega^\perp\bigr)}{\sqrt{1-t^2}}\\
& \;\;\;\;\; + \frac{g_1\bigl(t\omega - \sqrt{1-t^2} \omega^\perp\bigr)g_2\bigl(t\omega + \sqrt{1-t^2} \omega^\perp\bigr)}{\sqrt{1-t^2}}.\end{align*}
In particular, when $n=2$ and $g:\mathbb{S}^1\to [0,\infty)$, we have
\begin{equation}\label{e:n2Id}
(gd\sigma)\ast (gd\sigma )(2x) = c \frac{\1_{|x|<1}}{|x|\sqrt{1-|x|^2}} g(|x| e_x +  \sqrt{1-|x|^2} e_x^\perp)g(|x| e_x -  \sqrt{1-|x|^2} e_x^\perp),
\end{equation}
where $e_x=x/|x|$.

\subsection{Proof of Theorem \ref{t:wStein}}
By the tomography reduction \eqref{basicidea}, the proof of Theorem \ref{t:wStein} may be reduced to the following lemma. We clarify first that the cutoff function $\gamma_R$ in the statement of Theorem \ref{t:wStein} should be taken of the form $\gamma_R(x)=\gamma(x/R)$, where $\gamma(x)=\psi(x)^3$ for some nonnegative radially decreasing $\psi\in\mathcal{S}(\mathbb{R}^2)$ with Fourier support in the unit ball. As will become clear, this specific structure is imposed merely for technical convenience.
\begin{lemma}\label{l:Red23Nov}
For $R\gg1$, $\omega\in\mathbb{S}^1$, and $g:\mathbb{S}^1\to\mathbb{C}$,
\begin{align}\label{e:PW23Nov}
&\big\| (-\Delta_v)^\frac14 X(\gamma_R |\widehat{gd\sigma}|^2)(\omega,\cdot) \big\|_{ L^2_v(\langle \omega \rangle^\perp)}
\lesssim
BT_{1/R} (|g|_{1/R}^2, |g|_{1/R}^2)(\omega)^\frac12 + BT_{1/R} ( |g|_{1/R}^2, |g|_{1/R}^2)(\omega^\perp)^\frac12. \nonumber
\end{align}
\end{lemma}
\begin{proof}
By elementary considerations we may reduce to the situation where $g$ is nonnegative and symmetric in the sense that $g(-\cdot)=g$. By the rotation invariance of the expressions involved,
it suffices to handle the case $\omega = e_2$. Moreover, since $\gamma_R=\psi_R^3$,
$$\gamma_R |\widehat{gd\sigma}|^2=\psi_R \cdot \psi_R \widehat{gd\sigma} \cdot \psi_R\overline{\widehat{gd\sigma}}.$$ With this in mind,
by Plancherel's theorem we have
\begin{align*}
&\big\| (-\Delta_v)^\frac14 X(\gamma_R |\widehat{gd\sigma}|^2)(e_2,\cdot) \big\|_{ L^2_v(\langle e_2 \rangle^\perp)}^2
\sim
\int_{\mathbb{R}} |\eta| \left(\widehat{\psi}_R\ast \widehat{\psi}_R \ast  (gd\sigma)\ast  \widehat{\psi}_R \ast (gd\sigma) (\eta,0)\right)^2\, d\eta .
\end{align*}
First we claim that
\begin{equation}\label{e:Poi28Nov}
\widehat{\psi}_R \ast (gd\sigma)(x) \lesssim R g_{1/R}(x/|x|),\;\;\; x\in \mathbb{R}^2,
\end{equation}
where $g_{1/R}$ is the function $g$ mollified at scale $1/R$ using the Poisson kernel on $\mathbb{S}^1$.
To see this we write $x = r(\cos\theta,\sin\theta)$ and $y=(\cos\phi,\sin\phi)$, and observe that
$$
\widehat{\psi}_R \ast (gd\sigma)(x) = R^2 \int_{0}^{2\pi} \widehat{\psi}(R(r\cos\theta - \cos\phi, r\sin\theta - \sin\phi)) g(\cos\phi,\sin\phi)\, d\phi.
$$
Since $\widehat{\psi}$ is assumed to be radially decreasing, that is,  $\widehat{\psi}(x) = h(|x|)$ for some smooth function $h$ supported on $[0,1]$, we have that $$\widehat{\psi}_R \ast (gd\sigma)(x) \le \widehat{\psi}_R \ast (gd\psi)(x/|x|).$$ Since
$$\widehat{\psi}(R(\cos\theta - \cos\phi, \sin\theta - \sin\phi))  = h\bigl(4R\sin^2\bigl(\tfrac{\theta-\phi}{2}\bigr)\bigr) \lesssim R^{-1} p_{1-1/R}(\theta-\phi),
$$
where $p_{r}(\theta) = (1-r^2)/(1-2r\cos\theta +r^2)$ is the Poisson kernel, we conclude that
$$
\widehat{\psi}_R \ast (gd\sigma)(x) \le \widehat{\psi}_R \ast (gd\sigma)(x/|x|) \lesssim R g_{1/R}(x/|x|),
$$
which establishes \eqref{e:Poi28Nov}.
Using this, polar coordinates, and the assumption that $\psi$ has Fourier support in the unit ball, we have
\begin{align*}
\widehat{\psi}_R\ast (gd\sigma) \ast &\widehat{\psi_R}\ast (gd\sigma) (x) \\
\lesssim &
R^2 \int_{\mathbb{R}^2} \chi_{(1-R^{-1}, 1+ R^{-1})}(|x-y|)  \chi_{(1-R^{-1}, 1+ R^{-1})}(|y|)   g_{1/R}(e_{x-y}) g_{1/R}(e_y)\, dy \\
\lesssim &
R^{2} \int_{1-R^{-1}}^{1+R^{-1}} \chi_{(1-R^{-1}, 1+ R^{-1})}(|x-t\xi|) \int_{\mathbb{S}^1} g_{1/R}(e_{x-t\xi}) g_{1/R}(\xi)\, d\sigma(\xi) dt.
\end{align*}
Consequently, by further use of polar coordinates,
\begin{align*}
\widehat{\psi}_R\ast \widehat{\psi}_R\ast &(gd\sigma) \ast \widehat{\psi}_R\ast (gd\sigma) (z) \\
\lesssim&
R^2 \int_{\mathbb{R}^2} \int_{1-R^{-1}}^{1+R^{-1}} \widehat{\psi}_R(z- (x-t\xi) ) \int_{\mathbb{S}^1} \chi_{(1-R^{-1}, 1+ R^{-1})}(|x|) g_{1/R}(e_{x}) g_{1/R}(\xi)\, d\sigma(\xi) dtdx\\
\sim &
R^2 \int_{1-R^{-1}}^{1+R^{-1}} \int_{1-R^{-1}}^{1+R^{-1}} \int_{\mathbb{S}^1} \widehat{\psi}_R(z- (s\tilde{\xi}-t\xi) ) \int_{\mathbb{S}^1}  g_{1/R}(\tilde{\xi}) g_{1/R}(\xi)\, d\sigma(\xi) d\sigma(\tilde{\xi}) dsdt.
\end{align*}
Dominating $\widehat{\psi}_R$ pointwise by a suitable constant multiple of the function
$$
\Phi_{1/R}(x):=\frac{R^2}{(1+|Rx|^2)^N}
$$
for a suitably large natural number $N$, we have
\begin{align*}
\widehat{\psi}_R\ast \widehat{\psi}_R\ast &(gd\sigma) \ast \widehat{\psi}_R\ast (gd\sigma) (z) \\
\lesssim&
\int_{\mathbb{S}^1} \Phi_{1/R}(z- (\tilde{\xi}-\xi) ) \int_{\mathbb{S}^1}  g_{1/R}(\tilde{\xi}) g_{1/R}(\xi)\, d\sigma(\xi) d\sigma(\tilde{\xi}) \\
=&
\int_{\mathbb{R}^2} \Phi_{1/R}(z- \tilde{\xi} ) \int_{\mathbb{S}^1}  g_{1/R}(\tilde{\xi}-\xi) \delta(1-|\tilde{\xi}-\xi|^2) g_{1/R}(\xi)\, d\sigma(\xi) d\tilde{\xi}\\
=&
\Phi_R\ast (g_{1/R} d\sigma) \ast (g_{1/R} d\sigma)(z),
\end{align*}
where we have used the local constancy property of the function $\Phi_{1/R}$ at scale $1/R$.
Using the formula \eqref{e:n2Id} and the locally constant property of $g_{1/R}$ at scale $R^{-1}$ we obtain
\begin{align*}
 \widehat{\psi}_R\ast \widehat{\psi}_R\ast &(gd\sigma) \ast \widehat{\psi}_R\ast (gd\sigma) (z) \\
\lesssim&
\int_{\mathbb{R}^2} \Phi_{1/R}(z-y) \frac{\1_{|y|<1}}{|y|\sqrt{1-|y|^2}} g_{1/R}(|y|e_y + \sqrt{1-|y|^2}e_y^\perp) g_{1/R}(|y|e_y - \sqrt{1-|y|^2}e_y^\perp)\, dy \\
\sim&
\int_{\mathbb{R}^2} \Phi_{1/R}(z-y) \frac{\1_{|y|<1}}{|y|\sqrt{1-|y|^2}} \, dy g_{1/R}(|z|e_z + \sqrt{1-|z|^2}e_z^\perp) g_{1/R}(|z|e_z - \sqrt{1-|z|^2}e_z^\perp) \\
\sim&
\frac{\1_{|z|<1}}{(|z|+R^{-1}) \sqrt{1-|z|^2 + R^{-1}}} g_{1/R}(|z|e_z + \sqrt{1-|z|^2}e_z^\perp) g_{1/R}(|z|e_z - \sqrt{1-|z|^2}e_z^\perp).
\end{align*}
Consequently we conclude that
\begin{align*}
\big\| (-\Delta_v)^\frac14 X&(\gamma_R |\widehat{gd\sigma}|^2)(e_2,\cdot) \big\|_{ L^2_v(\langle e_2 \rangle^\perp)}^2 \\
\lesssim &
\int_{\mathbb{R}}  \frac{\1_{|\eta|<1}}{(|\eta|+R^{-1}) (1-|\eta|^2 + R^{-1})} g_{1/R}(|\eta|, \sqrt{1-|\eta|^2})^2 g_{1/R}(|\eta|, -\sqrt{1-|\eta|^2})^2\, d\eta \\
\lesssim&
\int_{\mathbb{S}^1} g_{1/R}(\xi)^2 \widetilde{g}_{1/R}(R_{e_1}(\xi))^2\, \frac{d\sigma(\xi)}{|\xi\cdot e_1| +R^{-1}} + \int_{\mathbb{S}^1} g_{1/R}(\xi)^2 \widetilde{g}_{1/R}(R_{e_2}(\xi))^2\, \frac{d\sigma(\xi)}{|\xi\cdot e_2| +R^{-1}},
\end{align*}
since $g$ is assumed to be symmetric.
\end{proof}
\subsection{Proof of Theorem \ref{t:wMizTak}}
As we observe in the introduction, \eqref{e:Reduced00} with $q=1$ follows immediately from Theorem \ref{t:main}, and hence it suffices to prove \eqref{e:Reduced}.
As we explain in the introduction, thanks to Theorem \ref{t:wStein}, the proof of Theorem \ref{t:wMizTak} may be reduced to establishing \eqref{e:BT1/2}.
\begin{proof}[Proof of \eqref{e:BT1/2}]
We assume, as we may, that $g_1,g_2$ are nonnegative and symmetric. Writing $\omega = (\cos\theta,\sin\theta)$ and $\xi = (\cos\varphi,\sin\varphi)$ in \eqref{e:BilinearT}, we have
$$
BT_\delta(g_1,g_2)(\omega) = \int_0^{2\pi} \frac{G_1(\varphi) G_2(2\theta-\varphi)}{|\cos(\theta-\varphi)| + \delta}\, d\varphi
= \int_0^{2\pi} \frac{G_1(\theta - \varphi) G_2(\theta+\varphi)}{|\cos\varphi| + \delta}\, d\varphi,
$$
where $G_i(\varphi) = g_i(\cos\varphi,\sin\varphi)$ for $i=1,2$.
After considering suitable rotations, the inequality \eqref{e:BT1/2} may be reduced to showing that
\begin{equation}\label{e:Goal23Nov2}
\int_0^{2\pi} \Bigg( \int_{10R^{-1}}^{1/100} \frac{h_1(\theta + \varphi) h_2(\theta - \varphi)} {\varphi}  \,  d\varphi \Bigg)^\frac12\, d\theta \lesssim \log(R) \|h_1\|_1^\frac12 \|h_2\|_1^\frac12.
\end{equation}
To do this we use a localisation argument of Kenig and Stein from their analysis of bilinear fractional integrals in \cite{KenigStein}.
Since we allow a logarithmic loss in $R$, and have
$$
\int_0^{2\pi} \Bigg( \int_{10R^{-1}}^{1/100} \frac{h_1(\theta + \varphi) h_2(\theta - \varphi)} {\varphi}  \,  d\varphi \Bigg)^\frac12\, d\theta
\lesssim
\sum_{k=0}^{\log(R)} \int_0^{2\pi} \Bigg( \dashint_{\varphi \sim 2^{-k}} h_1(\theta + \varphi) h_2(\theta - \varphi) \,  d\varphi \Bigg)^\frac12\, d\theta,
$$
it suffices to prove that
\begin{equation}\label{e:Goal23Nov3}
\int_0^{2\pi} \Bigg( \dashint_{\varphi \sim \lambda} h_1(\theta + \varphi) h_2(\theta - \varphi) \,  d\varphi \Bigg)^\frac12\, d\theta \lesssim \|h_1\|_1^\frac12\|h_2\|_1^\frac12,
\end{equation}
uniformly in the (dyadic) scale $0<\lambda<1$.
For each $\lambda$, we decompose $[0,2\pi] = \bigcup_{i=1}^{O(\lambda^{-1})} I_i(\lambda)$, where $I_i(\lambda) = [(i-1)\lambda,i\lambda]$, and use the Cauchy--Schwarz inequality to obtain
\begin{align*}
\int_0^{2\pi} \Bigg( \dashint_{\varphi \sim \lambda} h_1(\theta + \varphi) h_2(\theta - \varphi) \,  d\varphi \Bigg)^\frac12\, d\theta
&=
\sum_i \int_{I_i(\lambda)} \Bigg( \dashint_{\varphi \sim \lambda} h_1(\theta + \varphi) h_2(\theta - \varphi) \,  d\varphi \Bigg)^\frac12\, d\theta \\
&\le
\lambda^\frac12 \sum_i \Bigg( \int_{I_i(\lambda)} \dashint_{\varphi \sim \lambda} h_1(\theta + \varphi) h_2(\theta - \varphi) \,  d\varphi  d\theta \Bigg)^\frac12.
\end{align*}
Notice that $\theta\pm \varphi \in J_i(\lambda):= [(i-10)\lambda,(i+10)\lambda]$ whenever $\theta \in I_i(\lambda)$ and $\varphi \sim \lambda$, and so we conclude that
\begin{align*}
\int_0^{2\pi} \Bigg( \dashint_{\varphi \sim \lambda} h_1(\theta + \varphi) h_2(\theta - \varphi) \,  d\varphi \Bigg)^\frac12\, d\theta
\lesssim&
\sum_i \Bigg( \int_{0}^{2\pi} \int_0^{2\pi} h_1\cdot \1_{J_i(\lambda)}(\theta + \varphi) h_2\cdot \1_{J_i(\lambda)}(\theta - \varphi) \,  d\varphi  d\theta \Bigg)^\frac12\\
\sim&
\sum_i \big( \|h_1\cdot \1_{J_i(\lambda)}\|_{1} \|h_2\cdot \1_{J_i(\lambda)}\|_1 \big)^\frac12
\\
\lesssim&
\|h_1\|_1^\frac12\|h_2\|_1^\frac12,
\end{align*}
thanks to the almost disjointness of the intervals $J_i(\lambda)$.
\end{proof}

\textbf{Remark.}
The argument above raises the question of the validity of weighted inequalities of the form
\begin{equation*}
\int_0^{2\pi} \Bigg( \int_0^{2\pi} \frac{f_1(\theta-\varphi)f_2(\theta+\varphi)}{|\varphi| + \delta} \, d\varphi\Bigg)^\frac12\, w(\theta)d\theta \lesssim \log(\delta^{-1}) \Bigg( \int_0^{2\pi} |f_1| w\Bigg)^\frac12 \Bigg( \int_0^{2\pi} |f_2| w \Bigg)^\frac12.
\end{equation*}
Indeed, if this were true with  a weight of the form $\mathcal{S}w$, then \eqref{e:wStein0} would reduce to
$$
\int_{B_R} |\widehat{gd\sigma}|^2 w \lesssim \log(R) \int_{\mathbb{S}^1} |g|_{1/R}(\omega)^2 \mathcal{S}w(\omega)\, d\sigma(\omega),
$$
which is closer to the intended form of Stein's conjecture than \eqref{e:wStein0}. Perhaps more realistically one might look for a Fefferman--Stein type estimate of the form
$$
\int_0^{2\pi} \Bigg( \int_0^{2\pi} \frac{f_1(\theta-\varphi)f_2(\theta+\varphi)}{|\varphi| + \delta} \, d\varphi \Bigg)^\frac12\, w(\theta)d\theta \lesssim \log(\delta^{-1}) \Bigg( \int_0^{2\pi} |f_1| Mw\Bigg)^\frac12 \Bigg( \int_0^{2\pi} |f_2| Mw \Bigg)^\frac12,
$$
for an appropriate maximal operator $M$. We do not pursue this here, but note some closely-related results in the context of bilinear fractional integrals -- see, for example \cite{Moen}, and the references there.

\subsection{Results in dimensions $n\ge3$}
We conclude this section with a higher dimensional analogue of Theorem \ref{t:wMizTak}. From a technical point of view relating to finiteness, it will be a little more convenient here to include a slightly higher power of $-\Delta_v$ in our estimates, rather than insert a truncation factor inside the $X$-ray transform as we did for $n=2$.
\begin{theorem}\label{t:wMizTakHigh}
Let $n\ge3$ and $p_n=4(n-1)/(2n-3)$. Then the inequality
\begin{equation}\label{e:wMizTakHigh1}
\big\| (-\Delta_v)^{\frac12(1-\frac{n-1}2)+\varepsilon} X(|\widehat{gd\sigma}|^2) \big\|_{L^1_\omega L^2_{v}(\mathcal{M}_{1,n})} \lesssim_\varepsilon \|g\|_{L^{p}(\mathbb{S}^{n-1})}^2
\end{equation}
holds for all $\varepsilon>0$ if and only if $p\ge p_n$. 
Furthermore, 
\begin{equation}\label{e:wMizTakHigh}
\big\| (-\Delta_v)^{\frac12(1-\frac{n-1}2)+\varepsilon} X(|\widehat{gd\sigma}|^2) \big\|_{L^2_{\omega,v}(\mathcal{M}_{1,n})} \lesssim_\varepsilon \|g\|_{L^{p_n}(\mathbb{S}^{n-1})}^2.
\end{equation}
\end{theorem}

\textbf{Remark.}
As $c_n^{1/2}(-\Delta)^{1/4}X$ is an isometry,
$$
\big\| (-\Delta_v)^{\frac\alpha2} Xf \big\|_{L^2_{\omega,v}(\mathcal{M}_{1,2})} = c_n^{\frac{1}{2}} \big\| (-\Delta_x)^{\frac12(\alpha-\frac12)}f\big\|_{L^2(\mathbb{R}^n)}
$$
for all $\alpha\in\mathbb{R}$, and so
\eqref{e:wMizTakHigh} is equivalent to
\begin{equation}\label{e:RieszExt}
\big\| (-\Delta_x)^{-\frac{n-2}{4} +\varepsilon}|\widehat{gd\sigma}|^2 \big\|_{L^2_{\omega,v}(\mathcal{M}_{1,n})} \lesssim_\varepsilon \|g\|_{L^{p_n}(\mathbb{S}^{n-1})}^2.
\end{equation}
Using the boundedness of the Riesz potential $(-\Delta_x)^{-\frac{n-2}{4} +\varepsilon}$: $L^{p(\varepsilon)/2}(\mathbb{R}^n) \to L^2(\mathbb{R}^{n})$, where $p(\varepsilon) \searrow \frac{2n}{n-1}$ as $\varepsilon \to0$, the left-hand side of \eqref{e:wMizTakHigh} may be controlled by
$
\|\widehat{gd\sigma}\|_{L^{p(\varepsilon)}(\mathbb{R}^n)}^2.
$
Therefore if the restriction conjecture \eqref{restconj} is true then, we have
$$
\big\| (-\Delta_x)^{-\frac{n-2}4 +\varepsilon}|\widehat{gd\sigma}|^2 \big\|_{L^2_{\omega,v}(\mathcal{M}_{1,n})}
\lesssim_\varepsilon \|g\|_{L^{\frac{2n}{n-1}}(\mathbb{S}^{n-1})}^2
$$
for arbitrary small $\varepsilon>0$. Since $2n/(n-1) > p_n$, this bound is weaker than \eqref{e:RieszExt}, providing a further illustration of the improvements available to the composition of $X$ with $|\widehat{gd\sigma}|^2$.

Applying the tomography reduction \eqref{basicidea}, Theorem \ref{t:wMizTakHigh} implies the following.
\begin{corollary}
Let $n\ge3$. For every $\varepsilon>0$,
$$
\int_{\mathbb{R}^n} |\widehat{gd\sigma}|^2 w \lesssim_\varepsilon \big\| (-\Delta_v)^{\frac{n-1}4 - \varepsilon} Xw \big\|_{L^2_{\omega,v} (\mathcal{M}_{1,n})} \|g\|_{L^\frac{4(n-1)}{2n-3}(\mathbb{S}^{n-1})}^2.
$$
\end{corollary}
In particular, we have the following weak version of \eqref{e:wMizTak} with $q=2$:
$$
\int_{\mathbb{R}^n} |\widehat{gd\sigma}|^2 w \lesssim_\varepsilon \big\| (-\Delta_v)^{\frac{n-1}4 - \varepsilon} Xw \big\|_{L^\infty_\omega L^2_{v} (\mathcal{M}_{1,n})} \|g\|_{L^\frac{4(n-1)}{2n-3}(\mathbb{S}^{n-1})}^2.
$$

We prove Theorem \ref{t:wMizTakHigh} by first reducing it to a statement involving $BA_t$ defined by \eqref{e:Leia}. In what follows we write $\mathbb{S}^{n-2}_\omega = \mathbb{S}^{n-1}\cap \langle \omega \rangle^\perp$ and $d\sigma_{\mathbb{S}^{n-2}_\omega}(\xi) = \delta(1-|\xi|^2)\delta(\xi\cdot \omega) d\xi$.
\begin{lemma}\label{l:Reduce}
Let $n\ge3$ and $1\le q <\infty$. Then for $\varepsilon>0$,
\begin{align*}
\big\| (-\Delta_v)^{\frac12 (1- \frac{n-1}2) +\varepsilon} &X( |\widehat{gd\sigma}|^2) \big\|_{L^q_\omega L^2_v(\mathcal{M}_{1,n})}^q \\
\sim&
\int_{\mathbb{S}^{n-1}} \Bigg( \int_{\mathbb{S}^{n-2}_\omega} \int_{0}^1 BA_{t}(g,g)(u)^2\, t^{-1 + 2\varepsilon}dt d\sigma_{\mathbb{S}^{n-2}_\omega}(u) \Bigg)^\frac q2\, d\sigma(\omega),
\end{align*}
where the implicit constant depends only on $n$ and $q$.
\end{lemma}

\begin{proof}
We again suppose $g$ is symmetric.
By Plancherel's theorem on $\langle \omega \rangle^\perp$, we have
\begin{align*}
\big\| (-\Delta_v)^{\frac12 (1- \frac{n-1}2) + \varepsilon } &X(  |\widehat{gd\sigma}|^2) \big\|_{L^q_\omega L^2_v(\mathcal{M}_{1,n})}^q \\
=&
\int_{\mathbb{S}^{n-1}} \Bigg( \int_{\langle \omega \rangle^\perp} |\eta|^{-(n-3) + 2\varepsilon } (gd\sigma)  \ast  (gd\sigma ) (\eta)^2\, d\lambda_\omega(\eta) \Bigg)^\frac q2\, d\sigma(\omega).
\end{align*}
Using polar coordinates on $\langle \omega \rangle^\perp$ and the identity \eqref{e:conv}, we conclude that
\begin{align*}
\big\| (-\Delta_v)^{\frac12 (1- \frac{n-1}2) + \varepsilon }& X(  |\widehat{gd\sigma}|^2) \big\|_{L^q_\omega L^2_v(\mathcal{M}_{1,n})}^q \\
\sim &
\int_{\mathbb{S}^{n-1}} \Bigg( \int_{\mathbb{S}^{n-2}_\omega} \int_{0}^\infty t^{-(n-3)+ 2\varepsilon}  (gd\sigma)  \ast  (gd\sigma ) (t u)^2\, t^{n-2}dtd\sigma_{\mathbb{S}^{n-2}_\omega}(u) \Bigg)^\frac q2\, d\sigma(\omega) \\
\sim &
\int_{\mathbb{S}^{n-1}} \Bigg( \int_{\mathbb{S}^{n-2}_\omega} \int_{0}^1 BA_{t}(g,g)(u)^2\, t^{-1+2\varepsilon}dt d\sigma_{\mathbb{S}^{n-2}_\omega}(u) \Bigg)^\frac q2\, d\sigma(\omega).
\end{align*}
\end{proof}
\begin{proof}[Proof of Theorem \ref{t:wMizTakHigh}]
We begin with the sufficiency of the condition $p\geq p_n$. By Lemma \ref{l:Reduce} it suffices to show that
\begin{equation}\label{e:Goal24Nov}
\int_{\mathbb{S}^{n-1}} \int_0^1 BA_t(g,g)(\omega)^2 t^{-1+2\varepsilon}\, dtd\sigma(\omega) \lesssim_\varepsilon \|g\|_{p_n}^4,
\end{equation}
since $$\int_{\mathbb{S}^{n-1}}\int_{\mathbb{S}^{n-2}_\omega} F(u)\, d\sigma_{\mathbb{S}^{n-2}_\omega}(u)d\sigma(\omega) = \int_{\mathbb{S}^{n-1}} F(\omega)\, d\sigma(\omega).$$
By the Cauchy--Schwarz inequality,
$$
BA_t(g,g)(\omega) \le A_t(g^2)(\omega)^\frac12 A_{-t}(g^2)(\omega)^\frac12,
$$
and so, by a further use of the Cauchy--Schwarz inequality,
$$
\int_{\mathbb{S}^{n-1}} BA_t(g,g)(\omega)^2 \, d\sigma(\omega)
\lesssim
\int_{\mathbb{S}^{n-1}} A_t(g^2)(\omega)^2\, d\sigma(\omega).
$$
Next we recall from Lemma \ref{Chrot} that there is a $t_n>0$ such that
$$
\sup_{0<t<t_n} \|A_tg\|_{L^n(\mathbb{S}^{n-1})} \lesssim \|g\|_{\frac{n}{n-1}}.
$$
Further, it is straightforward to verify that for $n\geq 3$,
$$
\|A_tg\|_{L^1(\mathbb{S}^{n-1})} \lesssim \|g\|_1,
$$
uniformly in $0<t<1$.
Interpolating these two estimates, we have
$$
\sup_{0<t<t_n} \|A_tg\|_{L^2(\mathbb{S}^{n-1})} \lesssim \|g\|_{\frac{2(n-1)}{2n-3}},
$$
and so
$$
\sup_{0<t<t_n} \| BA_t(g,g)\|_{L^2(\mathbb{S}^{n-1})} \lesssim \|g\|_{p_n}^2.
$$
Hence by splitting the integral in \eqref{e:Goal24Nov}, we have
\begin{align*}
\int_{\mathbb{S}^{n-1}} \int_0^1 BA_t(g,g)&(\omega)^2 t^{-1+2\varepsilon}\, dtd\sigma(\omega)\\
&\lesssim
\int_0^{t_n} \|g\|_{p_n}^4t^{-1+2\varepsilon}\, dt + \int_{\mathbb{S}^{n-1}} \int_{t_n}^1 BA_t(g,g)(\omega)^2 t^{-1+2\varepsilon}\, dtd\sigma(\omega) \\
&\lesssim_\varepsilon \|g\|_{p_n}^4 + \int_{\mathbb{S}^{n-1}} \int_{0}^1 BA_t(g,g)(\omega)^2 \, dtd\sigma(\omega).
\end{align*}
To bound the second term above, we write
\begin{align*}
\int_{\mathbb{S}^{n-1}} \int_{0}^1 BA_t&(g_1,g_2)(\omega)^2 \, dtd\sigma(\omega) \\
=&
\int_{\mathbb{S}^{n-1}} \int_{\mathbb{S}^{n-1}\times \mathbb{S}^{n-1}} g_1(\xi) \widetilde{g_2}(R_\omega(\xi)) g_1(\eta) \widetilde{g_2}(R_\omega(\eta)) \delta((\xi -\eta)\cdot \omega )  \, d\sigma(\xi)d\sigma(\eta)d\sigma(\omega) \\
\lesssim&
\|g_2\|_\infty^2 \int_{\mathbb{S}^{n-1}\times \mathbb{S}^{n-1}} g_1(\xi) g_1(\eta)  |\xi-\eta|^{-1}  \, d\sigma(\xi)d\sigma(\eta) .
\end{align*}
Applying the Hardy--Littlewood--Sobolev inequality on the sphere (see for instance \cite{LiebLoss}), we obtain
$$
\int_{\mathbb{S}^{n-1}} \int_{0}^1 BA_t(g_1,g_2)(\omega)^2 \, dtd\sigma(\omega)
\lesssim
\|g_1\|_{\frac{p_n}{2}}^2\|g_2\|_{\infty}^2.
$$
Using the symmetry property
$
BA_t(g_1,g_2) = BA_t(g_2,g_1),
$
and bilinear interpolation, we conclude that
$$
\int_{\mathbb{S}^{n-1}} \int_{0}^1 BA_t(g_1,g_2)(\omega)^2 \, dtd\sigma(\omega)
\lesssim
\|g_1\|_{p_n}^2\|g_2\|_{p_n}^2,
$$
as required.

Finally we turn to the necessity of $p\ge p_n$. In view of Lemma \ref{l:Reduce}, it suffices to consider necessary conditions for the estimate
\begin{equation}\label{e:Red24Nov}
\int_{\mathbb{S}^{n-1}} \Bigg( \int_{\mathbb{S}^{n-2}_\omega} \int_{0}^1 BA_{t}(g,g)(u)^2\, t^{-1 + 2\varepsilon}dt d\sigma_{\mathbb{S}^{n-2}_\omega}(u) \Bigg)^\frac 12\, d\sigma(\omega) \lesssim_\varepsilon \|g\|_p^2,
\end{equation}
where $\varepsilon>0$ is arbitrary small. Applying this to the function $g = \1_{|(\xi_1,\ldots,\xi_{n-1})|\le \delta}$ we have
$$
\int_{\mathbb{S}^{n-1}} \Bigg( \int_{\mathbb{S}^{n-2}_\omega} \int_{0}^1 BA_{t}( \1_{|(\xi_1,\ldots,\xi_{n-1})|\le \delta}, \1_{|(\xi_1,\ldots,\xi_{n-1})|\le \delta})(u)^2\, t^{-1 + 2\varepsilon}dt d\sigma_{\mathbb{S}^{n-2}_\omega}(u) \Bigg)^\frac 12\, d\sigma(\omega) \lesssim_\varepsilon \delta^\frac{2(n-1)}p.
$$
Next we observe that for all $0<t<\delta$,
$$
BA_{t}( \1_{|(\xi_1,\ldots,\xi_{n-1})|\le \delta}, \1_{|(\xi_1,\ldots,\xi_{n-1})|\le \delta})(u) \sim A_0(\1_{|(\xi_1,\ldots,\xi_{n-1})|\le \delta}) (u) \sim \delta^{n-2} \1_{E_\delta}(u),
$$
where $E_\delta = \{ \xi \in\mathbb{S}^{n-1}:|\xi_n|\le \delta/10\}$. Hence the left hand side of \eqref{e:Red24Nov} is bounded from below by
$$
\delta^{n-2}\int_{\mathbb{S}^{n-1}} \Bigg( \int_{\mathbb{S}^{n-2}_\omega} \int_{0}^\delta \1_{E_\delta}(u)\, t^{-1 + 2\varepsilon}dt d\sigma_{\mathbb{S}^{n-2}_\omega}(u) \Bigg)^\frac 12\, d\sigma(\omega)
\sim \delta^{n-2} \delta^\varepsilon \int_{\mathbb{S}^{n-1}} \sigma_{\mathbb{S}^{n-2}_\omega}(E_\delta)^\frac12\, d\sigma(\omega).
$$
Since $n\ge3$, a simple geometrical observation reveals that
$$
\sigma_{\mathbb{S}^{n-2}_\omega}(E_\delta) \gtrsim \delta
$$
uniformly in $\omega\in\mathbb{S}^{n-1}$. This gives a lower bound of $\delta^{n-2 + \varepsilon} \delta^{1/2}$ for the left-hand side of \eqref{e:Red24Nov}. This implies that $\delta^{(n-3)/2 + \varepsilon} \lesssim \delta^{2(n-1)/p} $ for all $\delta$, and so $1/p \le 1/p_n + \varepsilon/(2n-2)$ for all $\varepsilon>0$, from which the necessity of $p\geq p_n$ follows.
\end{proof}

\section*{Appendix: endpoint Bloom--Sampson estimates}\label{Sec7}
Here we consider the validity of inequalities of the form
\begin{equation}\label{e:power-weight}
\big\|  \widehat{gd\sigma} \langle \cdot \rangle^{-\gamma} \big\|_{L^q(\mathbb{R}^n)} \lesssim \| g \|_{L^p(\mathbb{S}^{n-1})},
\end{equation}
and their Lorentz space variants; here
$p,q\ge1$ and $\gamma \in \mathbb{R}$. 
In particular we prove a general result, which upon specialising to $n=3$, implies Proposition \ref{sufprop}.
Of course when $\gamma=0$, this problem becomes the classical restriction problem \eqref{restconj}, and so a complete understanding of \eqref{e:power-weight} is not currently expected. However, if one restricts attention to $p\le 2$, then the complexity essentially amounts to that of the classical Stein-Tomas restriction theorem and the trace lemma. This was largely clarified by Bloom and Sampson in \cite{BS92}. Following their notation we distinguish the points
\[
A= \left(\tfrac12,\tfrac12\right),\;\;\; B=\left(\tfrac12,\tfrac{n-1}{2(n+1)}\right),\;\;\;C=\left(\tfrac12,0\right),\;\;\;D=(1,0),\;\;\;E=\left(1,\tfrac12\right)
\]
in $(\tfrac{1}{p},\tfrac{1}{q})$ space, noting that $A$ with $\gamma>1$ essentially corresponds to the trace lemma, and $B$ with $\gamma=0$ corresponds to the Stein--Tomas restriction theorem.
\begin{theorem}[\cite{BS92}]\label{t:BS}
Let $n\ge2$.
\begin{enumerate}
\item
If $(\frac1p,\frac1q) = A$, then \eqref{e:power-weight} holds if and only if $\gamma >\frac1q$.
\item
If $(\frac1p,\frac1q) \in {\rm int}\, BCD \cup (B,C)$, then \eqref{e:power-weight} holds if and only if $\gamma \ge 0$.
\item
If $(\frac1p,\frac1q) \in {\rm int}\, ADE \cup (A,E)$, then \eqref{e:power-weight} holds if and only if
$$
\gamma \ge \frac{n}{q} - \frac {n-1}{p'}.
$$
\item
If $(\frac1p,\frac1q)\in ABD \setminus \{A\}$, then \eqref{e:power-weight} holds if
$$
\gamma > \frac{n+1}{2q} - \frac{n-1}{2p'}
$$
and only if
$$
\gamma \ge \frac{n+1}{2q} - \frac{n-1}{2p'}.
$$
\end{enumerate}
\end{theorem}
In the above theorem the case $(\frac1p,\frac1q)\in ABD \setminus \{A\}$ with the critical power $\gamma = \tfrac{n+1}{2q} - \tfrac{n-1}{2p'}$ is clearly missing. Our main result in this section addresses this. In particular we establish this critical estimate on the interior of $ABD$, and prove a restricted weak type estimate on $(A,B) \cup (A,D)$. Our results are phrased in terms of the classical Lorentz spaces $L^{q,r}$, $0 < q,r \le \infty$.
\begin{theorem}\label{p:w-rest}
Let $n\ge 2$.
\begin{enumerate}
\item
If $(\frac1p,\frac1q) \in {\rm int}\, ABD \cup [B,D]$, then for all $r\in [1,\infty]$,
\begin{equation}\label{e:LorentzImproved}
\big\|  \widehat{gd\sigma} \langle \cdot \rangle^{-\gamma} \big\|_{L^{q,r}(\mathbb{R}^n)} \lesssim \| g \|_{L^{p,r}(\mathbb{S}^{n-1})}
\end{equation}
holds with
\begin{equation}\label{e:gamma}
\gamma = \frac{n+1}{2q} - \frac{n-1}{2p'}.
\end{equation}
In particular, \eqref{e:power-weight} holds with the same exponents.
\item
If $(\frac1p,\frac1q) \in (A,B) \cup (A,D)$, then
\begin{equation}\label{e:RestWeak}
\big\|  \widehat{gd\sigma} \langle \cdot \rangle^{-\gamma} \big\|_{L^{q,\infty}(\mathbb{R}^n)} \lesssim \| g \|_{L^{p,1}(\mathbb{S}^{n-1})}
\end{equation}
holds with $\gamma$ given by \eqref{e:gamma}.
\end{enumerate}
\end{theorem}
Some brief remarks are in order. First of all, for the purposes of deducing Proposition \ref{sufprop} it suffices to choose $r=2$ in \eqref{e:LorentzImproved} and use the embedding $L^p\subset L^{p,2}$, which holds as long as $p\le2$.
Second, if $(\frac1p,\frac1q) \in [B,D]$ then the resulting estimate is a consequence of the Stein--Tomas restriction theorem, and so we may restrict our attention to the region $ABD \setminus [B,D]$. Finally, setting $r=q$ in \eqref{e:LorentzImproved} and using the embedding $L^{p} \subset L^{p,q}$, which holds whenever $p\leq q$, yields \eqref{e:power-weight} on ${\rm int}\, ABD \cup [B,D]$ with the critical power \eqref{e:gamma}.

\begin{proof}[Proof of Theorem \ref{p:w-rest}] It is convenient to begin with Part {\it (2)}.
We first prove \eqref{e:RestWeak} for $(\frac1p,\frac1q) \in (A, D)$, where $q=p'$. Our goal is therefore to show that
\begin{equation}\label{e:0801-1}
\big\| \widehat{gd\sigma} \langle \cdot \rangle^{-\frac{1}{q}} \big\|_{L^{{q,\infty}}(\mathbb{R}^n)} \lesssim \| g \|_{L^{q',1}(\mathbb{S}^{n-1})}
\end{equation}
for $2 < q < \infty$.
From now on, we fix an arbitrary $q_*\in (2,\infty)$ and prove \eqref{e:0801-1} with $q=q_*$.
The first step is to write $$\mathbb{R}^n = \bigcup_{j=0}^\infty \mathcal{A}_j,$$ where
\[
\mathcal{A}_0 = B(0,1),\quad \mathcal{A}_j = B(0,2^{j+1}) \setminus B(0,2^j),
\]
and show that
\begin{equation}\label{e:local}
\big\| \widehat{gd\sigma} \langle \cdot \rangle^{-\frac1q} \1_{\mathcal{A}_j} \big\|_{L^{{q}}(\mathbb{R}^n)} \lesssim \| g \|_{L^{q'}(\mathbb{S}^{n-1})}
\end{equation}
for all $2\le {q} \le \infty$, uniformly in $j$. By analytic interpolation this will follow from the extreme cases $q=2$ and $q=\infty$. The latter follows immediately from the elementary estimate
$\| \widehat{gd\sigma} \|_\infty \lesssim \| g \|_1$. For $q=2$ we apply 
the weighted extension estimate \eqref{MizTak}, which is known for radial weights (see \cite{BRV,CS}),
with weight $w(x) = \langle x \rangle^{-1}\1_{\mathcal{A}_j}(x)$. This results in
\begin{align*}
\big\| \widehat{gd\sigma} \langle \cdot \rangle^{-\frac12} \1_{\mathcal{A}_j} \big\|_{L^{{2}}(\mathbb{R}^n)}
&\lesssim
\big\|X[\langle \cdot \rangle^{-1} \1_{\mathcal{A}_j}] \big\|_\infty \| g \|_2
\sim
2^{-j} \big\|X[\1_{\mathcal{A}_j}] \big\|_\infty \| g \|_2 \sim \| g \|_2,
\end{align*} uniformly in $j$, as required.
In order to use the estimates \eqref{e:local} to bound the sum in $j$, we use an argument of Bourgain \cite{Bourgain}, and in particular, Lemma 2.3 of Lee and Seo \cite{LeeSeo}.
Let us write
$$
\big\| \widehat{gd\sigma} \langle \cdot \rangle^{-\frac{1}{q_*}} \big\|_{L^{{q_*,\infty}}(\mathbb{R}^n)}  \sim \bigg\| \sum_{j} 2^{-\frac j{q_*}} \widehat{gd\sigma}   \1_{\mathcal{A}_j} \bigg\|_{L^{{q_*,\infty}}(\mathbb{R}^n)} =: \bigg\| \sum_j f_j \bigg\|_{L^{{q_*,\infty}}(\mathbb{R}^n)},
$$
and
choose $q_0, q_1$ satisfying $2< q_0 < q_* <q_1 < \infty$. By \eqref{e:local},
$$
\| f_j \|_{L^{q_i}(\mathbb{R}^n)} \sim 2^{j(\frac1{q_i} - \frac1{q_*})}  \big\| \widehat{gd\sigma} \langle \cdot \rangle^{-\frac1{q_i}} \1_{\mathcal{A}_j}  \big\|_{L^{q_i}(\mathbb{R}^n)} \lesssim 2^{j(\frac1{q_i} - \frac1{q_*})} \| g \|_{L^{q_i'}(\mathbb{S}^{n-1})},
$$
uniformly in $j$, for each $i=1,2$.
Since $\frac1q_1 - \frac1{q_*} < 0 < \frac1{q_0} - \frac1{q_*}$, by Lemma 2.3 of \cite{LeeSeo}, we conclude that
$$
\bigg\| \sum_j f_j \bigg\|_{L^{q_*,\infty}(\mathbb{R}^n)} \lesssim \| g \|_{L^{q_*,1}(\mathbb{S}^{n-1})}.
$$
This completes the proof of \eqref{e:0801-1} with $q=q_*$.

To complete the proof of Part {\it (2)}, we must also establish \eqref{e:RestWeak} on the segment $(A,B)$. However, this argument is similar to that for $(A,D)$ above, and so we leave the details to the reader.

We now turn to Part {\it (1)}. By Part {\it (2)} and complex interpolation, \eqref{e:RestWeak}  holds for all $(\frac1p,\frac1q) \in ABD\setminus\{A\}$ under the condition \eqref{e:gamma}. So our task is to improve \eqref{e:RestWeak} with respect to Lorentz exponents, and we do this using a real interpolation argument. For each $\gamma \in \mathbb{R}$ we define the line
$$
\ell(\gamma) = \left\{ (x,y) \in ABD: \gamma = \tfrac{n+1}{2}y - \tfrac{n-1}{2}(1-x) \right\}.
$$
Note  that \eqref{e:gamma} holds if $(\frac1p,\frac1q) \in \ell(\gamma)$, and that $\ell(0) = [B,D]$ and $\ell(\frac12) = \{A\}$. 
Fix a $\gamma_* \in (0,\frac12)$ and denote by  $T_{\gamma_*}$ the linear operator $g \mapsto \widehat{gd\sigma} \langle \cdot \rangle^{-\gamma_*}$. Since
\begin{equation}\label{e:goal0105}
\| T_{\gamma_*}g \|_{L^{q,\infty} (\mathbb{R}^n)} = \big\| \widehat{gd\sigma} \langle \cdot \rangle^{-\gamma_*}  \big\|_{L^{q,\infty}(\mathbb{R}^n)} \lesssim \| g \|_{L^{p,1}(\mathbb{S}^{n-1})}
\end{equation}
for all $(\frac1p,\frac1q) \in \ell(\gamma_*)$, real interpolation (see \cite{BerghLofstrom} for example) reveals that
$$
\| T_{\gamma_*}g \|_{L^{q,r}(\mathbb{R}^n)} \lesssim \| g \|_{L^{p,r}(\mathbb{S}^{n-1})}
$$
for all $(\frac1p,\frac1q) \in \ell(\gamma_*)$ and all $r\in [1,\infty]$. This establishes \eqref{e:LorentzImproved}. 
\end{proof}

\end{document}